\documentclass[a4paper,oneside,reqno]{amsbook}

\calclayout

\usepackage{etoolbox}

\usepackage[utf8]{inputenc}
\usepackage[T1]{fontenc}
\usepackage{lmodern}
\usepackage[english]{babel}
\usepackage{microtype} 

\usepackage{amsmath,amssymb,amsthm}
\usepackage{thmtools} 

\usepackage{mathtools} 
\usepackage{upgreek}
\usepackage{accents} 
\usepackage{stmaryrd} 

\usepackage{tikz} 
\usetikzlibrary{cd}
\usetikzlibrary{calc}

\usepackage{dsfont} 
\usepackage{pifont} 

\usepackage{ifthen}
\usepackage{enumitem}
\setlist[itemize]{labelindent=\parindent,leftmargin=*,itemsep=3pt,topsep=5pt}
\setlist[enumerate]{label=\textup{(\roman{*})},itemsep=3pt,labelindent=0pt,leftmargin=*}

\usepackage{nicematrix}

\usepackage{graphicx} 
\usepackage{subcaption}
\usepackage[hypertexnames=false]{hyperref} 
\hypersetup{hidelinks}
\usepackage[nameinlink,sort]{cleveref} 
\usepackage{orcidlink}

\usepackage{booktabs} 
\usepackage{esint} 

\usepackage{bbm}

\newcommand{\R}{\mathbb{R}} 
\newcommand{\N}{\mathbb{N}} 
\newcommand{\Z}{\mathbb{Z}} 
\newcommand{\C}{\mathbb{C}} 

\newcommand{\K}{\mathbb{K}} 

\newcommand{\Avscon}{A very short course on}

\newcommand{\Cc}[1][\infty]{\mathrm{C}_{\mathrm{c}}\ifthenelse{\equal{#1}{}}{}{^{#1}}}
\newcommand{\Lp}[2][]{\mathrm{L}_{#2\ifthenelse{\equal{#1}{}}{}{,#1}}} 
\newcommand{\Lb}{\mathcal{L}_{\mathrm{b}}} 
\newcommand{\sobH}{\mathrm{H}}
\newcommand{\cH}{\sobH_0}

\newcommand{\wlim}{\textnormal{w-}\lim}
\newcommand{\dd}{\mathrm{d}} 

\DeclareMathOperator{\diag}{diag}
\DeclareMathOperator{\ran}{ran}

\DeclareMathOperator{\supp}{spt}

\DeclareMathOperator{\dom}{dom}

\renewcommand{\div}{\operatorname{div}}
\DeclareMathOperator{\grad}{grad}
\DeclareMathOperator{\curl}{curl}
\newcommand{\divcon}{\mathop{\div_{-1}}}

\newcommand{\cgrad}{\operatorname{grad}_0}

\DeclarePairedDelimiter{\norm}{\lVert}{\rVert}
\DeclarePairedDelimiter{\abs}{\lvert}{\rvert}

\DeclarePairedDelimiterX{\dset}[2]{\{}{\}}{#1\,\delimsize\vert\,\mathopen{} #2}
\DeclarePairedDelimiterX{\scprod}[2]{\langle}{\rangle}{#1,#2}

\renewcommand{\Re}{\operatorname{Re}}

\newcommand{\weakto}{\rightharpoonup}

\newcommand{\Htopo}{\mathrm{H}}

\renewcommand{\leq}{\leqslant}

\renewcommand{\geq}{\geqslant}

\newtheoremstyle{myremark}{3pt}{3pt}{}{}{\itshape}{.}{ }{}
\newtheoremstyle{myplain}{3pt}{3pt}{\itshape}{}{\bfseries}{.}{ }{}
\newtheoremstyle{mydefinition}{3pt}{3pt}{}{}{\bfseries}{.}{ }{}

\declaretheorem[name=Definition,style=mydefinition,numberwithin=section,qed=$\sslash$]{defin}
\declaretheorem[name=Remark,style=myremark,numberlike=defin,qed=$\sslash$]{remark}
\declaretheorem[name=Example,numberlike=defin,style=myremark,qed=$\sslash$]{example}
\declaretheorem[name=Lemma,numberlike=defin,style=myplain]{lemma}
\declaretheorem[name=Corollary,numberlike=defin,style=myplain]{corollary}
\declaretheorem[name=Proposition,numberlike=defin,style=myplain]{proposition}
\declaretheorem[name=Theorem,numberlike=defin,style=myplain]{theorem}
\declaretheorem[name=Exercise,numberwithin=chapter,style=mydefinition]{exercise}
\declaretheorem[name=Take Homes,style=myremark]{takehomes}

\begin{document}

\title[the Schur topology]{\includegraphics{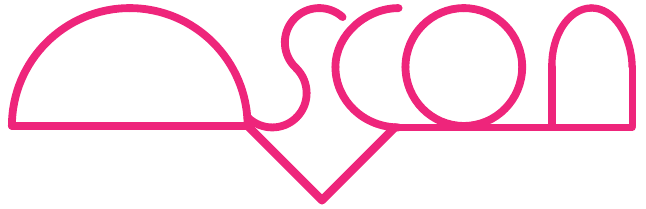}\\
 \Avscon \\ the Schur topology} 

\author[A.~Buchinger]{Andreas Buchinger\,\orcidlink{0009-0004-4203-5874}}
\email{andreas.buchinger@tuhh.de}

\address{Technische Universität Hamburg\\
  Institut für Mathematik\\
  Am Schwarzenberg-Campus 3\\
  D-21073 Hamburg\\
  Germany}

%
%

\author[M.~Waurick]{Marcus Waurick\,\orcidlink{0000-0003-4498-3574}}
\email{marcus.waurick@math.tu-freiberg.de}

\address{TU Bergakademie Freiberg \\
  Institute of Applied Analysis \\
  Akademiestrasse 6 \\
  D-09596 Freiberg \\
  Germany}

\date{\today}
\dedicatory{}




\begin{abstract} The aim of the course is to lead to an understanding of homogenisation processes in an operator-theoretic sense. In fact, using solely operator-theoretic means not referring to the particular form of the coefficients, we will identify an operator topology on the level of coefficients that will fully capture the convergence involved in the context of homogenisation. One upshot of this perspective will be that we will obtain homogenisation results for time-dependent partial differential equations (almost) for free.
  \end{abstract}

\maketitle%

\section*{Extended Abstract and Preliminaries}

In the course of the following three lectures, we shall introduce the concept of $\Htopo$-convergence, which is tailored to discuss homogenisation problems very generally without having to resort to periodic settings or any structural assumptions on the local behaviour of the coefficients. We will eventually apply our findings to two different partial differential equations. Firstly, to the equations describing the interconnected effects of heat and elasticity, that is, the equations of thermoelasticity. Secondly, we shall consider homogenisation problems for Maxwell's equations. These two applications will be postponed until the end of this mini-course so that everyone stays excited and motivated along the seemingly abstract way to our destination. Even though the findings might seem abstract, they are overall rather elementary and only rarely require tedious computations. In the following, we shall therefore rely on some well-known functional analytic concepts such as Hilbert and Banach spaces. We will heavily use the concept of weak convergence on Hilbert spaces, the fact that the closed unit ball is weakly (sequentially) compact. Moreover, for bounded linear operators, we will need the concepts of convergence in the strong and the weak operator topology. Furthermore, we will occasionally have to use weak-* (sequential) compactness of the unit ball in $\Lp{\infty}(\Omega)$ for open subsets $\Omega\subseteq \R^d$. As a consequence of the mentioned topological notions, we will also directly depend on the definition of initial topologies. On the operator-theoretic side, we will use the adjoint of bounded (and also unbounded) linear operators on Hilbert spaces; necessary results will be recalled during the lectures. Finally, the notion of compact operators should ring a bell. Throughout this course, an \emph{embedding}, always denoted by $\mathcal{X}\hookrightarrow \mathcal{Y}$, is the identity map from $\mathcal{X}$ into a possibly larger space $\mathcal{Y}\supseteq\mathcal{X}$. Scalar products are linear in the second and anti-linear in the first component. Weak convergence is denoted by $\rightharpoonup$, strong/norm convergence by $\to$, and, throughout the whole course, $d$ stands for an arbitrary fixed element of $\N$. Finally, the superscript $(k)$ for any $1\leq k\leq d$ stands for the $k$-th entry of a $d$-dimensional
vector or, applied to a set of $d$-dimensional vectors, for the image under the projection onto the $k$-th entry.

After each lecture, there will be a short summary section called \emph{take homes} where the most important messages of the lecture will be recalled. The lectures will each conclude with a short collection of exercises that help understand the appearing concepts.

\chapter{Lecture 1}

\section{Introduction}\label{sec:intro}

A homogenisation problem in mathematical physics seeks to approximate the coefficients of a partial differential equation with simpler ones that preserve the essential physical behaviour of the original system. 
Usually, this approximation process is realised by a limit process. This limit process is characterised by the (weak) convergence of solutions of problems with increasingly more complex coefficients to the solution of a limit equation with (hopefully) less complex coefficients. On the level of solution operators for the respective PDEs, the underlying convergence is characterised by convergence with respect to the weak operator topology. This short course is devoted to the description of the topology on the level of coefficients. As we shall see, this will also involve the weak operator topology; however in a somewhat more convoluted sense. In order to properly set the stage, we will focus on examples in this lecture and generalise later on. For this purpose, we will have occasion to revisit the Lax--Milgram lemma and to provide a version with a more explicit solution operator than common. 

\section{Lax--Milgram Revisited}\label{sec:LMR}

This section states and proves a more explicit form of the Lax--Milgram lemma. To this end, let $\mathcal{H}_0$ and $\mathcal{H}_1$ be Hilbert spaces with the common underlying field $\K\in \{\R;\C\}$, and let
\[
    C\colon \dom(C)\subseteq \mathcal{H}_0 \to \mathcal{H}_1
\]
be a \emph{linear mapping} (or, equivalently, \emph{linear operator}), i.e., $\dom(C)\subseteq \mathcal{H}_0$, the domain of definition of $C$,  is a linear subspace of
$\mathcal{H}_0$ and $C(x+\lambda y)=Cx+\lambda Cy$ holds for all $x,y\in \dom(C)$ and $\lambda \in \K$.
For such $C$, we can endow $\dom(C)$ with the inner product
\begin{equation*}
\langle \cdot,\cdot\rangle_{\dom(C)}\colon
\begin{cases}
\hfill \dom(C)\times \dom(C)&\to\quad \K\\
 \hfill (x,y) &\mapsto\quad \langle x,y\rangle_{\mathcal{H}_0}+\langle Cx,Cy\rangle_{\mathcal{H}_1}
 \end{cases}\text{,}
\end{equation*}
called the \emph{graph scalar product}; the corresponding norm is called \emph{graph norm}.
 We further ask $C$ to be \emph{densely defined}, that is, $\dom(C)$ to be dense in $\mathcal{H}_0$. Finally $C$ is assumed to be \emph{closed}, that means, 
$(\dom(C), \langle \cdot,\cdot\rangle_{\dom(C)})$ is assumed to be a Hilbert space.

Before we come to the announced Lax--Milgram type result, we provide some examples for $C$ first. In the following, we will use that, for any open $\Omega\subseteq \R^d$, the space $\Cc(\Omega)$ of arbitrarily often differentiable functions with compact support in $\Omega$, is dense in $\Lp{2}(\Omega)$.
\begin{example}\label{ex:grad0}
Let $\Omega\subseteq \R^d$ be open.
\begin{enumerate}
 \item Consider $f \in \Lp{2}(\Omega)$ and $g_k \in \Lp{2}(\Omega)$ for a $k\in \{1,\ldots, d\}$, and assume that, for all $\phi\in \Cc(\Omega)$,
\[
     - \int_\Omega f \partial_k \phi = \int_\Omega g_k \phi\text{.}
\] If $\tilde{g}_k\in \Lp{2}(\Omega)$ has the same property, we obtain $\langle g_k-\tilde{g}_k,\phi\rangle_{\Lp{2}(\Omega)}=0$
for all $\phi\in \Cc(\Omega)$. Since $\Cc(\Omega)$ is dense in $\Lp{2}(\Omega)$, we infer $\tilde{g}_k=g_k$.

If such a unique $g_k$ exists for each $k\in \{1,\ldots, d\}$, we call $f$ \emph{weakly differentiable} and write $\partial_k f\coloneqq g_k$ for $k\in \{1,\ldots, d\}$. The set of all weakly differentiable $f\in\Lp{2}(\Omega)$ is denoted by $\sobH^1(\Omega)$. We easily see that $\sobH^1(\Omega)$ is a linear subspace of $\Lp{2}(\Omega)$ and 
that $\partial_k$ is linear on $\sobH^1(\Omega)$ for each $k\in \{1,\ldots, d\}$ (\Cref{exer:1.0}).

\item Define
\[
    \grad \colon \begin{cases}
\hfill \sobH^1(\Omega)\subseteq \Lp{2}(\Omega)&\to\quad \Lp{2}(\Omega)^d\\
 \hfill f &\mapsto\quad (\partial_k f)_{k\in \{1,\ldots, d\}}
 \end{cases}\text{.}
\]
$\grad$ is densely defined, since $\Cc(\Omega)\subseteq \sobH^1(\Omega)$. Moreover, $\grad$ is closed: Indeed, let $(f_n)_{n\in\N}$ be a Cauchy-sequence in $\sobH^1(\Omega)$ with respect to the norm induced by the graph scalar product of $\grad$. Then, both $(f_n)_{n\in\N}$ as well as $((\partial_k f_n)_{k\in \{1,\ldots,d\}})_{n\in\N}$ are Cauchy sequences in $\Lp{2}(\Omega)$ and $\Lp{2}(\Omega)^d$, respectively. Let $f\in\Lp{2}(\Omega)$ and $(g_k)_{k\in \{1,\ldots,d\}}\in\Lp{2}(\Omega)^d$ be their respective limits. For $\phi\in \Cc(\Omega)$,
$k\in \{1,\ldots,d\}$ and $n\in\N$, we get
\[
     - \int_\Omega f_n \partial_k \phi = \int \partial_kf_n  \phi\text{.}
\]
Letting $n\to\infty$, we infer 
\[
     - \int_\Omega f \partial_k \phi = \int g_k  \phi
\] for all $k\in \{1,\ldots,d\}$ and $\phi\in\Cc(\Omega)$, which yields the claim. In other words, $\sobH^1(\Omega)$ is a Hilbert space (endowed with the
graph scalar product arising from $\grad$).

\item We define $\cH^1(\Omega)$ to be the closure of $\Cc(\Omega)$ in $\sobH^1(\Omega)$. As a consequence, the restriction of $\grad$ to $\cH^1(\Omega)$,
\[
    \cgrad \colon \begin{cases}
\hfill \cH^1(\Omega)\subseteq \Lp{2}(\Omega)&\to\quad \Lp{2}(\Omega)^d\\
 \hfill f &\mapsto\quad (\partial_k f)_{k\in \{1,\ldots, d\}}
 \end{cases}\text{.}
\]
is also closed. In other words, $\cH^1(\Omega)$ is a Hilbert space (endowed with the
graph scalar product arising from $\grad$). Due to $\Cc(\Omega)\subseteq \cH^1(\Omega)$, $\cgrad$ is also densely defined.
\end{enumerate}
\end{example}

We return to our general setting with general Hilbert spaces $\mathcal{H}_0$ and $\mathcal{H}_1$. For a bounded linear operator $A \in \Lb(\mathcal{H}_0,\mathcal{H}_1)$ we recall that its \emph{adjoint} $A^* \colon\mathcal{H}_1\to \mathcal{H}_0$
is, by the Riesz-representation theorem, well-defined via assigning to $y\in \mathcal{H}_1$ the unique element $A^*y \in \mathcal{H}_0$ characterised by 
\[
   \langle A^*y , x \rangle_{\mathcal{H}_0}  = \langle y, Ax \rangle_{\mathcal{H}_1}
\]
for all $x\in \mathcal{H}_0$.

\begin{lemma}\label{lem:adj0} Let $A \in \Lb(\mathcal{H}_0,\mathcal{H}_1)$. Then the following statements hold:
\begin{enumerate}
\item $A^*\in\Lb(\mathcal{H}_1,\mathcal{H}_0)$ with $\norm{A^*}=\norm{A}$.
\item \label{lem:adj0it2} $A^{**}=A$.
\item $\ker(A)=\{x\in \mathcal{H}_0\mid Ax = 0\}=\ran(A^*)^\perp = \{ x\in \mathcal{H}_0\mid \forall y \in \ran(A^*)\colon \langle x,y\rangle_{\mathcal{H}_0}=0\}$.
\item \label{lem:adj0it4}$\ker(A^*)^{\perp}=\overline{\ran}(A)$.
\end{enumerate}
\end{lemma}
\begin{proof}
This is \Cref{exer:1.1}.
\end{proof}

Let $C\colon \dom(C)\subseteq \mathcal{H}_0 \to \mathcal{H}_1$ be densely defined and closed.
From now on, whenever we just write $\dom (C)$, e.g., as in \Cref{ex:grad0} (ii) or (iii) when talking about $C=\grad$, $\mathcal{H}_0= \Lp{2}(\Omega)$, $\mathcal{H}_1= \Lp{2}(\Omega)^d$ with $\dom(C)=\sobH^1(\Omega)$ or $\dom(C)=\cH^1(\Omega)$, we consider $\dom(C)$ as a Hilbert space endowed with $\langle \cdot,\cdot\rangle_{\dom(C)}$.
Furthermore, we define the \emph{dual}
\[
C^\diamond  \colon
\begin{cases}
\hfill \mathcal{H}_1&\to\quad \dom(C)'\coloneqq\Lb(\dom(C),\K)\\
 \hfill q &\mapsto\quad (\dom(C)\ni \phi \mapsto \langle q,C\phi\rangle_{\mathcal{H}_1})
 \end{cases}\text{.}
\]
\begin{proposition}\label{prop:Cdiam} Let $C\colon \dom(C)\subseteq \mathcal{H}_0 \to \mathcal{H}_1$ be densely defined and closed. Define $\tilde{C}\colon \dom(C)\to \mathcal{H}_1, x\mapsto Cx$\footnote{Note that, in the definition of $\tilde{C}$, $\dom(C)$ is considered as a Hilbert space on its own right, the operator $\tilde{C}$ is by construction \emph{always} bounded. In contrast, as $C$ considered to be an operator from $\mathcal{H}_0$ to $\mathcal{H}_1$, may be unbounded--all interesting applications consider unbounded $C$.} Then, $\tilde{C}\in\Lb(\dom(C),\mathcal{H}_1)$ with the adjoint
\[
    {\tilde{C}}^*=R_{\dom(C)} C^\diamond\in \Lb(\mathcal{H}_1,\dom(C)),
\]
where  $R_{\dom(C)}\colon \dom(C)'\to \dom(C), q\mapsto \tilde{q}$ is the bijective and norm-preserving Riesz mapping finding to each functional the corresponding Riesz-representative.
\end{proposition}
\begin{proof}
$\tilde{C}\in\Lb(\dom(C),\mathcal{H}_1)$ immediately follows from $\norm{Cx}^2_{\mathcal{H}_1}\leq \norm{x}^2_{\mathcal{H}_0}+\norm{Cx}^2_{\mathcal{H}_1}$ for $x\in\dom(C)$.

 Let $q\in \mathcal{H}_1$. Then, we compute for $\phi\in \dom(C)$
 \[
     (C^\diamond q)(\phi) =  \langle q,C\phi\rangle_{\mathcal{H}_1}=\langle q,\tilde{C}\phi\rangle_{\mathcal{H}_1} = \langle \tilde{C}^* q, \phi\rangle_{\dom(C)} = (R_{\dom(C)}^{-1}( {\tilde{C}}^* q))(\phi). \qedhere
 \]
\end{proof}

The upshot of the latter proposition is that $C^\diamond$ modulo some Riesz mapping is the adjoint of \emph{bounded} linear operator. As a consiequence, results as in \Cref{lem:adj0} can be applied. A first instance of that is the following observation.

\begin{lemma}\label{lem:adj} Let $C\colon \dom(C)\subseteq \mathcal{H}_0 \to \mathcal{H}_1$ be densely defined and  closed. 
Then the following statements hold:
\begin{enumerate}
\item If $C$ has dense range, then $C^{\diamond}$ is one-to-one.
\item  If $C$ is one-to-one, then $C^\diamond$ has dense range.
\end{enumerate}
\end{lemma}
\begin{proof}
The result follows from \Cref{lem:adj0} in conjunction with \Cref{prop:Cdiam}.
\end{proof}

Another consequence of \Cref{prop:Cdiam} can be found in the proof of the following statement, where closedness of $\ran(C^\diamond)$ is characterised by the same for $\ran(C)$. In fact, in the proof, this result is reduced to the bounded operator case.

\begin{theorem}\label{thm:crt} Let $C\colon \dom(C)\subseteq \mathcal{H}_0 \to \mathcal{H}_1$ be densely defined and  closed. Then, $\ran(C)\subseteq \mathcal{H}_1$ is closed if and only if $\ran(C^\diamond)\subseteq \dom(C)'$ is closed.
\end{theorem}
\begin{proof}
  By \Cref{prop:Cdiam}, we can equivalently show that $\tilde{C}\in\Lb(\dom(C),\mathcal{H}_1)$ has closed range if and only if ${\tilde{C}}^*\in\Lb(\mathcal{H}_1,\dom(C))$ has closed range.
  Recalling~\labelcref{lem:adj0it2} from \Cref{lem:adj0}, it suffices to prove that, for a bounded linear operator $A\in \Lb(\mathcal{H}_0,\mathcal{H}_1)$, a closed $\ran(A)\subseteq \mathcal{H}_1$ implies that $\ran(A^*)\subseteq \mathcal{H}_0$ is closed.
  
  By \Cref{exer:1.2}, $\ran(A)\subseteq \mathcal{H}_1$ is closed if and only if there exists $\gamma>0$ such that
  \[
      \gamma \norm{ x}_{\mathcal{H}_0}\leq \norm{Ax}_{\mathcal{H}_1}
  \]
  for all $x\in\ker(A)^\perp$, and a similar statement holds for $A^*$.
  By closedness of $\ran(A)\subseteq \mathcal{H}_1$,~\labelcref{lem:adj0it4} from \Cref{lem:adj0} reads 
  $\ker(A^*)^\perp={\ran(A)}$. We let $y = Ax$ for some $x\in \ker(A)^\bot$, $x\neq0$ and compute
  \begin{align*}
      \norm{ A^*y}_{\mathcal{H}_0} & = \sup_{\substack{w\in \mathcal{H}_0\\ \norm{w}_{\mathcal{H}_0}\leq 1}} \abs{\langle A^*y, w \rangle_{\mathcal{H}_0}} \\
       & =  \sup_{\substack{w\in \mathcal{H}_0\\ \norm{w}_{\mathcal{H}_0}\leq 1}} \abs{\langle y, A w \rangle_{\mathcal{H}_1}}      =  \sup_{\substack{w\in \ker(A)^\perp\\ \norm{w}_{\mathcal{H}_0}\leq 1}} \abs{\langle y, A w \rangle_{\mathcal{H}_1}}  \\
   &   =  \sup_{\substack{w\in \ker(A)^\perp\\ \norm{w}_{\mathcal{H}_0}\leq 1}} \abs{\langle Ax , A w \rangle_{\mathcal{H}_1}} \\
   & =  \frac{1}{\norm{x}_{\mathcal{H}_0}} \norm{Ax}_{\mathcal{H}_1}^2\geq \gamma \norm{Ax}_{\mathcal{H}_1} = \gamma \norm{y}_{\mathcal{H}_1}\text{,}
  \end{align*}
  proving the claim.  
\end{proof}
The solution theory for abstract elliptic type equations can now be provided as follows.

\begin{theorem}\label{thm:TW14} Let $C\colon \dom(C)\subseteq \mathcal{H}_0 \to \mathcal{H}_1$ be densely defined,  closed, one-to-one, and onto. Let $a\in \Lb(\mathcal{H}_1,\mathcal{H}_1)$. 

Then, the following conditions are equivalent:
\begin{enumerate}
 \item\label{thm:TW14i} For all $f\in \dom(C)'$, there exists a unique $u\in \dom(C)$ such that
 \[
     \forall\phi \in \dom(C):\langle a Cu, C\phi \rangle_{\mathcal{H}_1} = f(\phi)\text{,}
 \]
 \item $ a $ is continuously invertible.
\end{enumerate}
In case either of the above holds, the unique element $u$ in~\labelcref{thm:TW14i} is given by
\[
    u = C^{-1} a^{-1}(C^{\diamond})^{-1}f\text{.}
\]
\end{theorem}
\begin{proof}
  Using the definition of $C^\diamond$,~\labelcref{thm:TW14i} is equivalent to
  \[
     \forall f\in \dom(C)' \exists ! u\in \dom(C)\colon C^\diamond a \tilde{C} u = f.
  \] Note that the statement in \labelcref{thm:TW14i} does not require considering $C$ as an unbounded operator. Only the action of this operator is needed and, hence, in the latter equation $ \tilde{C}$ and $C$ can be used interchangeably. By \Cref{prop:Cdiam}, $\tilde{C}$ is a bounded linear operator. 
  Considering the assumptions, the open mapping theorem yields that $\tilde{C}$ also has a bounded inverse.
  \Cref{lem:adj} and \Cref{thm:crt} show that $C^\diamond$ is bijective. By \Cref{exer:1.35}, both $C^\diamond$ and its inverse are continuous.
  Hence, as $\tilde{C}$ and $C^\diamond$ are both topological isomorphisms, unique existence of $u$ follows if and only if $a$ is continuously invertible. In this case, the solution formula 
  \[
  u = \tilde{C}^{-1} a^{-1}(C^{\diamond})^{-1}f
  \]follows. Again, using that $C$ and $\tilde{C}$ act the same way, one can again replace the one inverse by the other eventually leading to the claimed formula.
\end{proof} 

\begin{corollary}\label{cor:TW14} Let $C\colon \dom(C)\subseteq \mathcal{H}_0 \to \mathcal{H}_1$ be densely defined,  closed with closed range. Let $a\in \Lb(\mathcal{H}_1,\mathcal{H}_1)$. 

Then, the following conditions are equivalent:
\begin{enumerate}
 \item\label{cor:TW14i} for all $f\in (\dom(C)\cap \ker(C)^\bot)'$ there exists a unique $u\in \dom(C)\cap \ker(C)^\bot$ such that
 \[
    \forall\phi \in \dom(C)\cap \ker(C)^\perp: \langle a Cu, C\phi \rangle_{\mathcal{H}_1} = f(\phi)\text{,}
 \]
 \item $ \iota^* a \iota $ is continuously invertible, where $\iota\colon \ran(C)\hookrightarrow \mathcal{H}_1$ is the canonical embedding.
\end{enumerate}
In case either of the above holds, the unique element $u$ in~\labelcref{cor:TW14i} is given by
\[
    u = (\iota^*C\restriction_{\dom(C)\cap \ker(C)^\perp})^{-1} (\iota^*a\iota)^{-1}((C\restriction_{\dom(C)\cap \ker(C)^\perp})^{\diamond}\iota )^{-1}f\text{.}
\]
\end{corollary}
\begin{proof}
Considering \Cref{exer:1.375}, we apply \Cref{thm:TW14} to $\iota^* C\restriction_{\dom(C)\cap \ker(C)^\perp}$, and use that 
 \[
     \forall\phi \in \dom(C)\cap \ker(C)^\perp :\langle a Cu, C\phi \rangle_{\mathcal{H}_1} = f(\phi)
 \]
 is equivalent to 
  \[
     \forall\phi \in \dom(C)\cap \ker(C)^\perp :\langle \iota^* a \iota \iota^*C\restriction_{\dom(C)\cap \ker(C)^\perp}u, \iota^*C\restriction_{\dom(C)\cap \ker(C)^\perp}\phi \rangle_{\mathcal{H}_1} = f(\phi)\text{,}
 \]
 and that $(\iota^*C)^{\diamond}=C^{\diamond}\iota$.
\end{proof}

\begin{remark}\label{rem:topiso}
Note that if, in addition to the assumptions in \Cref{cor:TW14} $C$ is one-to-one, then $\iota^*\tilde{C} \colon \dom(C)\to \ran(C)$ is a topological isomorphism. Moreover, since $(\iota^*C)^{\diamond}=C^{\diamond}\iota$ is then bijective (\Cref{lem:adj} and \Cref{thm:crt}), we deduce that $C^\diamond \iota \colon \ran(C)\to \dom(C)'$ is also a topological isomorphism.
\end{remark}

\section{The domain of the gradient}

The leading examples in most of the variational problems that will be discussed in this course will involve the gradient operator with Dirichlet boundary conditions, i.e., $\cgrad=\grad\restriction_{\cH^1(\Omega)}$. Thus, we shall now prove closedness of the range of this gradient operator such that we can obtain
well-posedness of the corresponding elliptic partial differential equation (\Cref{cor:TW14}). Indeed, the next statement does imply closedness of the range invoking \Cref{exer:1.2}.
\begin{proposition}\label{prop:Poincare}
Let $\Omega\subseteq \R^d$ be open and bounded. Then there exists $\gamma>0$ such that for all $u \in \cH^1(\Omega)$
\[ 
  \gamma \norm{ u}_{\Lp{2}(\Omega)}\leq \norm{\cgrad u}_{\Lp{2}(\Omega)^d}\text{.}
\]
\end{proposition}
\begin{proof}
It suffices to prove the statement for $\phi\in \Cc(\Omega)$ instead of $u$ since $C_c^\infty(\Omega)$ is, by definition, dense in $\cH^1(\Omega)$ with respect to
$(\norm{\cdot}^2_{\Lp{2}(\Omega)}+\norm{\grad\cdot}^2_{\Lp{2}(\Omega)^d})^{1/2}$. Thus, let $\phi\in \Cc(\Omega)$, extended by $0$ outside of $\Omega$. By the boundedness $\Omega$, there exists $R>0$ such that $\abs{x_1}\leq R$ whenever $(x_1,\ldots,x_d)\in \Omega$. Thus, the Cauchy--Schwarz inequality yields
\begin{align*}
  \int_{\Omega} |\phi(x)|^2\dd x & =  \int_\R \cdots \int_\R\int_{-R}^R |\phi(x_1, \ldots, x_d)|^2\dd x_1\ldots \dd x_d  \\
  & = \int_\R \cdots \int_\R\int_{-R}^R  \big|\int_{-R}^{x_1} \partial_1 \phi (s,x_2,\ldots,x_d)\dd s\big|^2\dd x_1\ldots \dd x_d \\
   & \leq \int_\R \cdots \int_\R\int_{-R}^R  \big(\int_{-R}^{R} |\partial_1 \phi (s,x_2,\ldots,x_d)|\dd s\big)^2\dd x_1\ldots \dd x_d \\
     & \leq 2R \int_\R \cdots \int_\R\int_{-R}^R  \int_{-R}^{R} |\partial_1 \phi (s,x_2,\ldots,x_d)|^2\dd s\dd x_1\ldots \dd x_d \\
& \leq (2R)^2 \int_{\Omega}|\grad \phi(x)|^2 \dd x\text{.}\qedhere
\end{align*}
\end{proof}

In the following lines, we shall specialise to one spatial dimension. In fact, the concept of $\Htopo$-convergence that will be introduced in the next section can be fully understood in one dimension using the results to follow. 
\begin{lemma}\label{lem:derOfTestFct}
Let $a<b$ and $\phi \in \Cc(a,b)$. Then, $\phi =\psi'$ for some $\psi\in \Cc(\Omega)$ if and only if $\int_a^b\phi =0$. If either case holds, $\psi$ is uniquely determined by
$\psi(x)\coloneqq \int_a^x \phi(y)\dd y$ for $x\in (a,b)$.
\end{lemma}
\begin{proof}
By the fundamental theorem of calculus, the necessity of $\int_a^b\phi =0$ is clear. On the other hand,
 let $a<c< d< b$ such that $\supp\phi\subseteq [c,d]$. Again, by the fundamental theorem, 
  $\int_c^d \phi =0$ implies $\psi'=\phi$ and $\psi\in \Cc(a,b)$ with $\supp\psi\subseteq [c,d]$ for $\psi$ given by $\psi(x)\coloneqq \int_a^x \phi(y)\dd y$, $x\in (a,b)$.
  The uniqueness of $\psi$ is clear.
\end{proof}
\begin{lemma}\label{lem:kerd} Let $a<b$ and $f\in \sobH^1(a,b)$ with $\partial f =0$. Then $f=c$ for some $c\in \K$.
\end{lemma}
\begin{proof}
   Let $\rho\in \Cc(a,b)$ with $\int_a^b\rho=1$, and define $c\coloneqq \int_a^b f\rho$. Then, for $\phi\in \Cc(a,b)$,
  \begin{equation*}
    \int_a^b \phi (f-c) = \int_a^b \phi f - \int_a^b \phi \int_a^b f\rho = \int_a^b \Bigl(\phi- (\int_a^b\phi)\rho\Bigr)f\text{.}
  \end{equation*}
  Since $\int_a^b (\phi- (\int_a^b\phi)\rho)=0$, we can apply \Cref{lem:derOfTestFct} and obtain some $\psi\in \Cc(a,b)$ with $\phi- (\int_a^b\phi)\rho =\psi'$, and, due to $\partial f=0$, that
  \[
  \forall \phi\in \Cc(a,b): \int_a^b \phi (f-c) = \int_a^b \psi' f = -\int_a^b \psi\partial f=0\text{.}
\qedhere  \]
\end{proof}
\begin{lemma}\label{lem:rand}
 Let $a<b$ and $g \in \Lp{2}(a,b)$. Then, $f$ given by $f(x)\coloneqq \int_a^x g(s)\dd s$ for $x\in (a,b)$ satisfies $f\in \sobH^1(a,b)$ and $g=\partial f$.
\end{lemma}
\begin{proof}
First, using the Cauchy--Schwarz inequality, we compute
\[\int_a^b\abs{f(x)}^2\dd x=\int_a^b \Bigl\lvert\int_a^x g(s)\dd s\Bigr\rvert^2\dd x\leq \int_a^b (b-a)\int_a^b\abs{g(s)}^2\dd s\dd x=(b-a)^2 \norm{g}_{\Lp{2}(a,b)}\text{.}
\]
  Secondly, by virtue of Fubini's theorem, we have
  \begin{multline*}
-  \int_a^b f(x) \phi'(x)\dd x = -  \int_{a}^b \int_a^x g(s)\dd s \phi'(x)\dd x \\ = 
 - \int_{a}^b \int_s^b g(s) \phi'(x) \dd x \dd s =  \int_{a}^b g(s) \phi(s) \dd s
  \end{multline*} for all $\phi\in \Cc(a,b)$, yielding the claim.
\end{proof}

\begin{theorem}\label{thm:1dSob}
Let $a<b$ and $f\in \sobH^1(a,b)$. Then, the following statements hold:
\begin{enumerate}
 \item\label{thm:1dSob1} $f$ admits a (unique) continuous representative in $\mathrm{C}[a,b]$.
\item\label{thm:1dSob2} $\int_a^x \partial f(s)\dd s =f(x)-f(a)$ for all $x\in [a,b]$ and the representative from~\labelcref{thm:1dSob1}.
\end{enumerate}
\end{theorem}
\begin{proof}
\labelcref{thm:1dSob1}: Define $g\coloneqq \partial f$. By \Cref{lem:rand}, $\tilde{f}$ given by $\tilde{f}(x)\coloneqq \int_a^x g(x)\dd x$ for $x\in (a,b)$ belongs to $\sobH^1(a,b)$, and $\partial \tilde{f} = g=\partial f$. Hence, by \Cref{lem:kerd}, there is $c\in \K$ such that 
 \begin{equation}\label{eq:ffc}
    f =\tilde{f} +c\text{.}
 \end{equation}
 Moreover, by Lebesgue's dominated convergence theorem, $\tilde{f}$ is continuous and can be uniquely continuously extended to $\{a,b\}$, which proves the assertion.
 
 \labelcref{thm:1dSob2}: By~\labelcref{eq:ffc}, we deduce
 \[
     f(x) =  \int_a^x g(x)\dd x + c
 \]
for all $x\in [a,b]$. Thus, $c=f(a)$.
\end{proof}
Using~\labelcref{thm:1dSob1} from \Cref{thm:1dSob}, we shall henceforth always choose the (unique)  continuous representative of $f\in \sobH^1(a,b)$, making point evaluations on $[a,b]$ well-defined.
\begin{corollary}\label{cor:H01} Let $f\in \sobH^1(a,b)$. Then $f\in \cH^1(a,b)$ if and only if $f(a)=f(b)=0$. 
\end{corollary}
\begin{proof}
$\Rightarrow\colon$ If $f\in \cH^1(a,b)$, we find $(\phi_n)_{n\in\N}$ in $\Cc(a,b)$ such that $\phi_n\to f$ in $\sobH^1(a,b)$ as $n\to\infty$. Thus, by~\labelcref{thm:1dSob2} from \Cref{thm:1dSob}, we have
\begin{equation}\label{eq:pwc}
   \phi_n (x)= \int_a^x \phi_n'(y)\dd y =\langle \mathds{1}_{[a,x]}, \phi_n'\rangle_{\Lp{2}(a,b)} \to \langle \mathds{1}_{[a,x]}, \partial f\rangle_{\Lp{2}(a,b)} = f(x)-f(a)
\end{equation}
for all $x\in (a,b]$ as $n\to\infty$. In particular, $0=\phi_n(b)\to f(b)-f(a)$, that is, $f(b)=f(a)$. It is left to show $f(a)=0$. For this, by \eqref{eq:pwc}, $(\phi_n)_{n\in\N}$ pointwise converges to $f-f(a)$ on $ [a,b]$. Next, $(\phi_n)_{n\in\N}$ is equicontinuous: Indeed, 
$\norm{\phi'_n}_{\Lp{2}(a,b)}\to \norm{\partial f}_{\Lp{2}(a,b)}$ as $n\to\infty$ and, by the Cauchy--Schwarz inequality,
\[
\abs{\phi_n(y)-\phi_n(x)}\leq \int_x^y\abs{\phi'_n(s)}\dd s\leq (y-x)^{1/2}\norm{\phi'_n}_{\Lp{2}(a,b)}
\]for $a\leq x < y\leq b$ and $n\in\N$. 
By the Arzelà--Ascoli theorem, $\phi_n\to f-f(a)$ in $C[a,b]$, and, thus, $\phi_n\to f-f(a)$ in $\Lp{2}(a,b)$. From $\phi_n\to f$ in $\Lp{2}(a,b)$ as $n\to\infty$, it thus follows $f(a)=0$. 

$\Leftarrow\colon$ Let $f\in \sobH^1(a,b)$ with $f(a)=f(b)=0$. Then,  as $\Cc(a,b)$ is dense in $\Lp{2}(a,b)$, we find $({\psi}_n)_{n\in\N}$ in $\Cc(a,b)$ such that ${\psi}_n \to \partial f$ in $\Lp{2}(a,b)$ as $n\to\infty$. Choose $\rho\in \Cc(a,b)$ with $\int_a^b\rho=1$ and define
\[
   \phi_n (x)\coloneqq \int_{a}^x \Bigl( \psi_n(s)-\rho(s)\int_a^b \psi_n(r)\dd r\Bigr)\dd s
\] 
for $n\in \N$ and $x\in (a,b)$.
Then, $\phi_n \in \Cc(a,b)$ for $n\in\N$ by \Cref{lem:derOfTestFct} and, by~\labelcref{thm:1dSob2} from \Cref{thm:1dSob},
 \[
 \int_a^b \psi_n(r)\dd r \to \int_a^b \partial f(r)\dd r = f(b) -f(a)=0
 \]
as $n\to\infty$.  From that, we infer $\phi_n' = \psi_n -\rho \int_a^b \psi_n(r)\dd r \to \partial f-\rho \int_a^b \partial f(r)\dd r =\partial f$ in  $\Lp{2}(a,b)$ as $n\to \infty$.
Furthermore, the Cauchy--Schwarz inequality and $f(a)=0$ yield
\begin{align*}
\norm{\phi_n-f}_{\Lp{\infty}(a,b)}&\leq \int_a^b\abs{\psi_n(s)-\partial f(s)}\dd s + \norm{\rho}_{\Lp{1}(\Omega)}\int_a^b \psi_n(r)\dd r\\
&\leq (b-a)^{1/2} \norm{\psi_n -\partial f}_{\Lp{2}(a,b)}+ \norm{\alpha}_{\Lp{1}(\Omega)}\int_a^b \psi_n(r)\dd r
\end{align*}
for $n\in\N$, which, in particular, implies $\phi_n\to f$ in $\Lp{2}(a,b)$ as $n\to \infty$.
\end{proof}

\begin{lemma}\label{lem:comg0} Let $a<b$ and set $\R\supseteq\Omega\coloneqq (a,b)$. Then,
 \[
 \mathfrak{g}_0(\Omega)\coloneqq  \ran(\cgrad) =\{1\}^\perp\subseteq \Lp{2}(\Omega)\text{.}
 \] Moreover, with $\iota\colon \mathfrak{g}_0(\Omega)\hookrightarrow \Lp{2}(\Omega)$, the othogonal projection $\iota^*\phi$ of $\phi\in \Lp{2}(\Omega)$ onto $\mathfrak{g}_0(\Omega)$ is given by 
 \[
 \iota^*\phi = \phi - \frac{1}{(b-a)^{1/2}}\langle 1,\phi\rangle_{\Lp{2}(\Omega)}\text{.}
 \]
\end{lemma}
\begin{proof}
If $\phi=\partial \psi$  for some $\psi\in \cH^1(\Omega)$,~\labelcref{thm:1dSob2} from \Cref{thm:1dSob} yields $\langle 1,\partial\psi\rangle_{\Lp{2}(\Omega)}=\psi(1)-\psi(0)=0$. On the other hand, let $\phi \in \Lp{2}(\Omega)$ with $\langle 1,\phi\rangle_{\Lp{2}(\Omega)} =0$. By \Cref{lem:rand}, we find $\psi\in \sobH^1(\Omega)$ such that $\partial \psi =\phi$ and
$0= \psi(a) =\psi(b)$. We deduce $\psi \in \cH^1(\Omega)$ by \Cref{cor:H01}, eventually showing $\mathfrak{g}_0(\Omega)  =\{1\}^\perp$. The remaining formula is elementary.
\end{proof}
\begin{lemma}\label{lem:reformsop}
 Let $a \in \Lp{\infty}(\Omega)$ such that $a^{-1}\in \Lp{\infty}(\Omega)$ and $\langle a^{-1}\rangle \coloneqq \langle 1,a^{-1}\rangle_{\Lp{2}(\Omega)}\neq 0$, where $\Omega\coloneqq (0,1)$. If $\iota\colon \mathfrak{g}_0(\Omega)\hookrightarrow \Lp{2}(\Omega)$, then
\[
    \forall \phi\in \mathfrak{g}_0 (\Omega): (\iota^*a \iota)^{-1} \phi = a^{-1} \phi-  a^{-1} \langle 1,a^{-1}\phi\rangle_{\Lp{2}(\Omega)} \frac{1}{\langle a^{-1}\rangle}\text{.}
\]
\end{lemma}
\begin{proof}
By \Cref{lem:comg0}, for all $\phi \in \Lp{2}(\Omega)$, $\iota^* \phi =\phi - \langle1, \phi\rangle_{\Lp{2}(\Omega)}$.

Next, let $\psi, \phi\in \mathfrak{g}_0$ with
\[
\iota^*a \iota  \psi =\phi.
\]
Then $\phi = a \psi -\langle1, a\psi\rangle_{\Lp{2}(\Omega)}$. Hence, $a^{-1}\phi = \psi -\langle 1, a\psi\rangle_{\Lp{2}(\Omega)} a^{-1}$ and taking the scalar product with $1$ recalling $\psi \in \mathfrak{g}_0(\Omega)=\{1\}^\bot$, we deduce
\[
   \langle 1,a^{-1}\phi\rangle_{\Lp{2}(\Omega)} =- \langle 1,a\psi\rangle_{\Lp{2}(\Omega)} \langle 1,a^{-1}\rangle_{\Lp{2}(\Omega)}\text{,}
   \]
    i.e.,
    \[
     \langle1, a\psi\rangle_{\Lp{2}(\Omega)} = -\frac{ \langle1, a^{-1}\phi\rangle_{\Lp{2}(\Omega)}}{ \langle1, a^{-1}\rangle_{\Lp{2}(\Omega)}}\text{.}
     \]
Thus,
\[
  (\iota^*a \iota)^{-1}  \phi =\psi = a^{-1}\phi+\langle 1,a\psi\rangle_{\Lp{2}(\Omega)} a^{-1} =a^{-1}\phi-\frac{ \langle1, a^{-1}\phi\rangle_{\Lp{2}(\Omega)}}{ \langle1, a^{-1}\rangle_{\Lp{2}(\Omega)}} a^{-1}.\qedhere\]
\end{proof}

\section{$\Htopo$-convergence}

Before we get to define $\Htopo$-convergence, we provide an elementary result first.

For a bounded linear operator $a\in\Lb(\mathcal{H})$ on a Hilbert space $\mathcal{H}$ we define $\Re a \coloneqq (a+a^*)/2$.
If we write $\Re a\geq c$ for some $c>0$, we mean
\[
\forall \phi\in\mathcal{H}: \Re \langle \phi, a\phi \rangle_{\mathcal{H}}=\langle \phi, \Re(a)\phi \rangle_{\mathcal{H}}\geq c \langle \phi, \phi \rangle_{\mathcal{H}}=c\norm{\phi}^2_{\mathcal{H}}\text{.}
\]

\begin{lemma}\label{lem:pda} Let $\mathcal{H}$ be a Hilbert space, $\mathcal{V}\subseteq \mathcal{H}$ a closed subspace and $a\in \Lb(\mathcal{H})$. If $\Re a \geq c>0$ for some $c>0$, then
\[
   a_{\mathcal{V}}\coloneqq \iota_{\mathcal{V}}^* a\iota_{\mathcal{V}} \in \Lb(\mathcal{V})
\]
is continuously invertible with $\norm{a_{\mathcal{V}}^{-1}}\leq 1/c$.
\end{lemma}
\begin{proof}
Note that 
\[ 
2\Re a_{\mathcal{V}} = \big( \iota_{\mathcal{V}}^* a\iota_{\mathcal{V}}+ \iota_{\mathcal{V}}^* a^*\iota_{\mathcal{V}}\big) = \big( \iota_{\mathcal{V}}^* (a+a^*)\iota_{\mathcal{V}}\big)\geq 2c \text{.}
\]
Hence, it suffices to prove the theorem for $\mathcal{V}=\mathcal{H}$. Next, for $\phi\in \mathcal{H}$,\[
      c \norm{\phi}_{\mathcal{H}}^2\leq \Re \langle \phi,a\phi\rangle_{\mathcal{H}} \leq \norm{\phi}_{\mathcal{H}}\norm{a \phi}_{\mathcal{H}}\text{.}\]
      Hence, $a$ is one-to-one and, by \Cref{exer:1.2}, it has closed range. Since $\Re a=\Re a^*$, $a^*$ is also one-to-one, and, by \Cref{lem:adj0}, $a$ has dense range, which
      altogether implies $a$ being onto. The norm bound for the inverse follows from $c \norm{\phi}_{\mathcal{H}}\leq \norm{a \phi}_{\mathcal{H}}$ for $\phi \in \mathcal{H}$. 
\end{proof}

For our definition of $\Htopo$-convergence, we introduce the following subset of bounded matrix-valued functions, which we will later on consider as multiplication operators on $\Lp{2}$-vector fields: let $0<\alpha\leq\beta$ and $\Omega\subseteq\R^d$ bounded and open. Then,
\[
  M(\alpha,\beta;\Omega)\coloneqq \{a\in \Lp{\infty}(\Omega)^{d\times d}: \Re a\geq \alpha, \Re a^{-1}\geq 1/\beta\}.
\]
We say that $(a_n)_{n\in\N}$ in $M(\alpha,\beta;\Omega)$ \emph{$\Htopo$-converges} to $a\in M(\alpha,\beta;\Omega)$, if for all $f\in \sobH^{-1}(\Omega)\coloneqq \cH^1(\Omega)'$ and $u_n\in \cH^1(\Omega)$ solving
\[
   \forall \phi\in \cH^1(\Omega): \langle a_n \cgrad u_n ,\cgrad \phi\rangle_{\Lp(\Omega)^d} = f(\phi)\text{,}
\]we get $u_n\rightharpoonup u\in \cH^1(\Omega)$ and $a_n\cgrad u_n\rightharpoonup a\cgrad u\in \Lp{2}(\Omega)^d$, where $u \in \cH^1(\Omega)$ is the solution of 
\[
    \forall \phi\in \cH^1(\Omega):\langle a \cgrad u ,\cgrad \phi\rangle_{\Lp(\Omega)^d} = f(\phi)\text{.}
\]

\begin{remark}Note that the variational problems used in this definition are yielding well-defined solutions based on \Cref{cor:TW14} and \Cref{rem:topiso}. Indeed, the Poincar\'e inequality \Cref{prop:Poincare} tells us that $\cgrad$ is one-to-one and has closed range $\mathfrak{g}_0(\Omega)$. Furthermore, $a\in M(\alpha,\beta;\Omega)$ implies $\Re a\geq \alpha$ and $a\in \Lb(\Lp{2}(\Omega)^d)$. Hence, by \Cref{lem:pda}, $\iota^* a \iota$ is continuously invertible, where $\iota \colon \mathfrak{g}_0(\Omega) \hookrightarrow L_2(\Omega)^d$ is the canonical embedding.
\end{remark}

\begin{theorem}[Tartar--Murat]\label{thm:Hbasic} Let $\Omega\subseteq \R^d$ open and bounded, $0<\alpha\leq\beta$,  $(a_n)_{n\in\N}$, $a$ in $M(\alpha,\beta;\Omega)$. Then the following statements hold:
\begin{enumerate}
 \item[(a)] There exists a metric topology $\tau_{\Htopo}$ on $M(\alpha,\beta;\Omega)$ such that $(a_n)_{n\in\N}$ $\Htopo$-converges to $a$, if and only if $(a_n)_{n\in\N}$ $\tau_{\Htopo}$-converges to $a$.
 \item[(b)] $(M(\alpha,\beta;\Omega),\tau_{\Htopo})$ is sequentially compact.
\end{enumerate}
\end{theorem}

Next, we may present a characterisation of $\Htopo$-convergence by means of convergences associated to $a$ instead of the solution operators of the variational problems.
\begin{theorem}\label{thm:varprob} Let $\Omega\subseteq \R^d$ open and bounded, $0<\alpha\leq\beta$, $(a_n)_{n\in\N}$, $a$ in $M(\alpha,\beta;\Omega)$, $\iota\colon \mathfrak{g}_0(\Omega)\hookrightarrow L_2(\Omega)^d$. Then, the following conditions are equivalent:
\begin{enumerate}
  \item $(a_n)_{n\in\N}$ $\Htopo$-converges to $a$.
  \item $(\iota^*a_n \iota)^{-1} \to (\iota^*a \iota)^{-1} $ and $a_n\iota(\iota^*a_n \iota)^{-1} \to a\iota(\iota^*a \iota)^{-1}$ as $n\to\infty$ in the weak operator topology of the respective spaces $\Lb(\mathfrak{g}_0(\Omega))$ and $\Lb(\mathfrak{g}_0(\Omega),\Lp{2}(\Omega)^d)$.
\end{enumerate}
\end{theorem}
We will obtain a proof in the course of proving the main result in the next lecture.

Using \Cref{thm:varprob}, we can show the following characterisation of $\Htopo$-convergence in one dimension.
\begin{proposition}\label{prop:qdind} Let $\Omega\coloneqq (0,1)$,  $0<\alpha\leq\beta$, $(a_n)_{n\in\N}$, $a$ in $M(\alpha,\beta;\Omega)$ and $\iota \colon \mathfrak{g}_0\hookrightarrow L_2(\Omega)$. Then, the following conditions are equivalent:
\begin{enumerate}
  \item\label{prop:qdind1} $a_n^{-1}\to a^{-1}$ as $n\to\infty$ in the weak operator topology.
  \item\label{prop:qdind2} $(\iota^*a_n \iota)^{-1} \to (\iota^*a \iota)^{-1} $ and $a_n\iota(\iota^*a_n \iota)^{-1} \to a\iota(\iota^*a \iota)^{-1}$ as $n\to\infty$ in the weak operator topology of the respective spaces $\Lb(\mathfrak{g}_0(\Omega))$ and $\Lb(\mathfrak{g}_0(\Omega),\Lp{2}(\Omega))$.
\end{enumerate}
\end{proposition}
\begin{proof}
We recall the representation of $(\iota^*a_n \iota)^{-1}$ from \Cref{lem:reformsop}: we have
\[
   (\iota^*a_n \iota)^{-1}\phi =a_n^{-1} \phi-  a_n^{-1} \langle 1,a_n^{-1}\phi\rangle_{\Lp{2}(\Omega)} \frac{1}{\langle a_n^{-1}\rangle}
\]
for $n\in\N$ and $\phi\in \mathfrak{g}_0(\Omega)$.
Thus, we immediately read off that~\labelcref{prop:qdind1} is sufficient for~\labelcref{prop:qdind2}. On the other hand, if~\labelcref{prop:qdind2} holds, we obtain 
\[
 \phi - \langle1, a_n^{-1}\phi\rangle_{\Lp{2}(\Omega)} \frac{1}{\langle a_n^{-1}\rangle} =  a_n\iota(\iota^*a_n \iota)^{-1}\phi
 \rightharpoonup a\iota(\iota^*a \iota)^{-1}\phi = \phi- \langle1, a^{-1}\phi\rangle_{\Lp{2}(\Omega)} \frac{1}{\langle a^{-1}\rangle}
\]
for all $\phi\in \mathfrak{g}_0(\Omega)$.
Since $b_n \coloneqq\langle a_n^{-1}\rangle ,n\in\N,$ defines a bounded sequence, we may choose a subsequence such that $b_{n_k} \to b$ for some $b\in \K$.
Then, by decomposing $\phi\in \Lp{2}(\Omega)$ into a $\mathfrak{g}_0(\Omega)$-function and a constant function (\Cref{lem:comg0}), we see that
\[
\langle1, a_{n_k}^{-1}\phi\rangle_{\Lp{2}(\Omega)}\to \langle1, a^{-1}\phi\rangle_{\Lp{2}(\Omega)}\frac{b}{\langle a^{-1}\rangle}\quad (\phi\in \Lp{2}(a,b))
\]
as $k\to\infty$.
 Since we can choose $\phi\in \Lp{2}(\Omega)$ to be any product of two $\Cc(\Omega)$-functions, it follows that
 \[
\langle\psi, a_{n_k}^{-1}\phi\rangle_{\Lp{2}(\Omega)}\to \langle\psi, a^{-1}\phi\rangle_{\Lp{2}(\Omega)}\frac{b}{\langle a^{-1}\rangle} 
 \]
 for all $\phi,\psi\in\Cc(\Omega)$ as $n\to \infty$, and thus,
\[
a_{n_k}^{-1} \to \frac{b}{\langle a^{-1}\rangle}a^{-1}
\]
in the weak operator topology as $k\to \infty$. From~\labelcref{prop:qdind1}$\implies$\labelcref{prop:qdind2}, we get that
\[
   (\iota^* a_{n_k}\iota)^{-1} \to    \bigl(\iota^* \frac{\langle a^{-1}\rangle}{b} a \iota\bigr)^{-1}
\]in the weak operator topology, which together with our assumption~\labelcref{prop:qdind2} yields $b=\langle a^{-1}\rangle$. Thus, in particular, a contradiction argument shows that $(b_n)_n$ converges to $\langle a^{-1}\rangle$. Thus, \labelcref{prop:qdind1} follows.
\end{proof}

\Cref{prop:qdind} is a purely one-dimensional phenomenon. In fact, in higher dimensions, we shall prove the following result in the next lecture.

\begin{theorem}\label{thm:hcon2d} Let $\Omega\subseteq \R^2$ be open and bounded, $0<\alpha\leq\beta$, and $(a_n)_{n\in\N}$ in $M(\alpha,\beta;\Omega)$ such that there exist $(\alpha_n)_{n\in\N}$, $\alpha_{\mathrm{h}}$, $\alpha_{\mathrm{m}}$ in $\Lp{\infty}(\Omega^{(1)})$ with $a_{n}(x) =\alpha_{n}(x_1) 1_{2\times 2}$ for $n\in\N$ and $x=(x_1,x_2)\in \Omega$. Then, the following conditions are equivalent:
\begin{enumerate}
\item $(a_n)_{n\in\N}$ $\Htopo$-converges to $\diag (\alpha_{\mathrm{h}}, \alpha_{\mathrm{m}})$.
\item $\alpha_n^{-1} \to \alpha_{\mathrm{h}}^{-1}$ and $\alpha_n\to \alpha_{\mathrm{m}}$ in $\sigma(\Lp{\infty}(\Omega^{(1)}),\Lp{1}(\Omega^{(1)}))$ as $n\to\infty$.
\end{enumerate}
\end{theorem}

This result shows that the behaviour of homogenisation in higher dimensions is not as easy as one could have possibly anticipated judging from the one-dimensional situation. In the next lecture, we will dive deep into the proof of the latter theorem as well as in the operator-theoretic description of $\Htopo$-convergence.

\section{Comments}

The abstract elliptic solution theory and its proofs presented in this lecture were originated in \cite{TrWa}, see also \cite[Section~2]{Wa18}. \Cref{thm:crt} is also known as the closed range theorem and can be found in \cite[Theorem IV.1.2]{Goldberg2006}. The fundamental observations concerning operators and their adjoints are rather standard and can also be found in \cite{Goldberg2006} or, e.g., as part of a so-called functional analysis tool box in \cite{PaZu23}.
The proofs of \Cref{prop:Poincare} to \Cref{lem:kerd} inclusive stem from \cite{ISem18}.
$\Htopo$-convergence was introduced by Murat and Tartar in the 1970s. Detailed information about this homogenisation theory in general, the particular role of periodic coefficients and, in particular, a proof of \Cref{thm:Hbasic},
can be found in, e.g., \cite{BLP78,MuTa97,CiDo99,ZKO94,Ta09}. Finally, \Cref{thm:varprob}, i.e., the operator-theoretic description of $\Htopo$-convergence solely addressing
properties of the coefficients, is the first fundamental result from \cite{Wa18} that we presented in this course.

\begin{takehomes}
\begin{enumerate}
 \item[(a)] Elliptic problems can be solved if the range of the (differential) operators is closed.
 \item[(b)] Elliptic problems are basically trying to find inverses of projected variants of the conductivity.
 \item[(c)] $\Htopo$-convergence addresses the continuous dependence of solution operators of elliptic equations on the coefficients
 \item[(d)] In one dimension, $\Htopo$-convergence is the same as convergence of the inverses of the coefficients in the weak operator topology.
 \item[(e)] Higher dimensions are different.
 \end{enumerate}
\end{takehomes}

\section{Exercises}
\begin{exercise}\label{exer:1.0}
Let $\Omega\subseteq \R^d$ be open. Show that $\sobH^1(\Omega)$ is a linear subspace of $\Lp{2}(\Omega)$ and 
that $\partial_k$ is linear on $\sobH^1(\Omega)$ for each $k\in \{1,\ldots, d\}$.
\end{exercise}
\begin{exercise}\label{exer:1.1} Prove \Cref{lem:adj0}.
\end{exercise}

\begin{exercise}\label{exer:1.2} Let $\mathcal{H}_0,\mathcal{H}_1$ be Hilbert spaces and $A\in \Lb(\mathcal{H}_0,\mathcal{H}_1)$. Show that the following conditions are equivalent:
\begin{enumerate}
\item $\ran(A)\subseteq \mathcal{H}_1$ is closed.
\item $\gamma(A)\coloneqq \inf\big\{ \tfrac{\norm{Ax}_{\mathcal{H}_1}}{\norm{x}_{\mathcal{H}_0}}: x\in \ker(A)^\bot\setminus\{0\}\big\}>0$.
\end{enumerate}
\end{exercise}
\begin{exercise}\label{exer:1.35}
Let $\mathcal{H}_0,\mathcal{H}_1$ be Hilbert spaces and let $C\colon \dom(C)\subseteq \mathcal{H}_0 \to \mathcal{H}_1$ be densely defined and closed. Show
that $C^\diamond$ is continuous. Furthermore, show: If $C^\diamond$ is invertible, then its inverse is also continuous.
\end{exercise}
\begin{exercise}\label{exer:1.375}
Let $\mathcal{H}_0,\mathcal{H}_1$ be Hilbert spaces and let $C\colon \dom(C)\subseteq \mathcal{H}_0 \to \mathcal{H}_1$ be densely defined, closed with closed range.
Show that $\dom(C)\cap \ker(C)^\perp$ is a closed subspace of $\dom(C)$, i.e., that $\dom(C)\cap \ker(C)^\perp$ endowed with (the restriction of) $\langle \cdot,\cdot\rangle_{\dom(C)}$ is a Hilbert space.

Furthermore, if $\iota\colon \ran(C)\hookrightarrow \mathcal{H}_1$ denotes the canonical embedding, show that $\iota \iota^* \in\Lb(\mathcal{H}_1)$ is the orthogonal projection onto $\ran (C)$, i.e., $\iota^*\in\Lb (\mathcal{H}_1,\ran(C))$ is the
orthogonal projection with its codomain restricted to its image, and show that $(\iota^*C)^{\diamond}=C^{\diamond}\iota$.
\end{exercise}
\begin{exercise}\label{exer:1.3} Show the following statement: Let $(\Omega,\mu)$ be a $\sigma$-finite measure space. Let $(\alpha_n)_n$, $\alpha$ in $L_\infty(\Omega)$. Then the following conditions are equivalent:
\begin{enumerate}
\item[(i)] $\alpha_n \to \alpha$ in $\sigma(L_\infty(\Omega),L_1(\Omega))$, that is, in $L_\infty(\Omega)$-weak-*.
\item[(ii)] Identifying any bounded measurable function with the corresponding multiplication operator in $L_2(\Omega)$, $\alpha_n\to \alpha$ in the weak operator topology.
\end{enumerate}
\end{exercise}

\chapter{Lecture 2}

\section{Introduction}

This chapter contains a proof of the following result.

\begin{theorem}\label{thm:hcon2d-L2} Let $\Omega\subseteq \R^2$ be open and bounded, $0<\alpha\leq\beta$, and $(a_n)_{n\in\N}$ in $M(\alpha,\beta;\Omega)$ such that there exist $(\alpha_n)_{n\in\N}$, $\alpha_{\mathrm{h}}$, $\alpha_{\mathrm{m}}$ in $\Lp{\infty}(\Omega^{(1)})$ with $a_{n}(x) =\alpha_{n}(x_1) 1_{2\times 2}$ for $n\in\N$ and $x=(x_1,x_2)\in \Omega$. Then, the following conditions are equivalent:
\begin{enumerate}
\item\label{thm:hcon2d-L21} $(a_n)_{n\in\N}$ $\Htopo$-converges to $\diag (\alpha_{\mathrm{h}}, \alpha_{\mathrm{m}})$.
\item\label{thm:hcon2d-L22} $\alpha_n^{-1} \to \alpha_{\mathrm{h}}^{-1}$ and $\alpha_n\to \alpha_{\mathrm{m}}$ in $\sigma(\Lp{\infty}(\Omega^{(1)}),\Lp{1}(\Omega^{(1)}))$ as $n\to\infty$.
\end{enumerate}
\end{theorem}

Already in the previous lecture, we implicitly used the so-called subsequence principle once in a while. Since it is of fundamental importance in the theory of homogenisation, we shall state and prove this principle explicitly next.

\begin{proposition}\label{prop:subseq} Let $X$ be a topological space, $(x_n)_{n\in\N}, x\in X$. Then the following conditions are equivalent:
\begin{enumerate}
  \item\label{prop:subseq1} $x_n \to x$ as $n\to\infty$.
    \item\label{prop:subseq2} every subsequence of $(x_n)_{n\in\N}$ has a subsequence converging to $x$.
\end{enumerate}
In particular, if $X$ is sequentially compact and every convergent subsequence of $(x_n)_{n\in\N}$ converges to $x$, then~\labelcref{prop:subseq1} holds.
\end{proposition}
\begin{proof}
\labelcref{prop:subseq1}$\Rightarrow$\labelcref{prop:subseq2} is obvious.
For the other direction, assume $x_n \nrightarrow x$. Then, there exist $U\subseteq X$ open with $x\in U$ and for all $n\in \N$ some $k>n$ such that $x_k\notin U$. Define
$\pi\colon\mathbb{N}\to\mathbb{N}$ recursively  such that $\pi(n)>\pi(n-1)$,  $x_{\pi(n)} \notin U$, and $x_{\pi(1)} = x_1$. By assumption, $(x_{\pi(n)})_{n\in\N}$ has a subsequence converging to $x$, which is a contradiction because $x_{\pi(n)}\notin U$ for all $n\in \N$.

Let $X$ be sequentially compact. Then, every subsequence of $(x_n)_{n\in\N}$ has a convergent subsequence. By assumption, this convergent subsequence converges to $x$. Hence,~\labelcref{prop:subseq2} holds and implies~\labelcref{prop:subseq1}.
\end{proof}

In order to prove \Cref{thm:hcon2d-L2}, some preparations are necessary. We begin with yet another fundamental observation in the theory of homogenisation: the $\div$-$\curl$ lemma.

\section{The $\div$-$\curl$ lemma}

For $\Omega\subseteq \R^d$, we say that $\Omega$ has \emph{continuous boundary}, if, locally, $\partial\Omega$ can be written as the graph of a continuous function; we refer to the sources in the comments for the details. In fact, we will not explicitly use said definition in any proof to come. The only thing we are interested in here is the following compactness result.

\begin{theorem}[Rellich--Kondrachov selection theorem]\label{thm:RKT} Let $\Omega\subseteq \R^d$ be open and bounded.

\begin{enumerate}
\item\label{thm:RKT1} Then, the embedding $\cH^1(\Omega)\hookrightarrow \Lp{2}(\Omega)$ is compact.

\item\label{thm:RKT2} If, in addition, $\Omega$ has continuous boundary, then the embedding $\sobH^1(\Omega)\hookrightarrow \Lp{2}(\Omega)$ is compact. In particular, $\ran(\grad)\subseteq \Lp{2}(\Omega)^d$ is closed.
\end{enumerate}
\end{theorem}

\begin{remark}\label{rem:comclos} The closedness of the range of an operator given a certain embedding is compact can be proved using a contradiction argument. The reader is asked to do the details in \Cref{exer:2.1}.
\end{remark}

\begin{remark}\label{rem:fafact}
Let $(B_n)_{n\in\N}$ be a sequence in $\Lb(\mathcal{X},\mathcal{Y})$, $\mathcal{X},\mathcal{Y}$ Banach spaces, $(x_n)_n$ in $\mathcal{X}$. If $(B_n)_{n\in\N}$ converges in the strong operator topology to some $B\in\Lb(\mathcal{X},\mathcal{Y})$ and $(x_n)_{n\in\N}$ to some $x$ in $\mathcal{X}$, then $B_n x_n \to Bx$ in $\mathcal{Y}$. Indeed, the uniform boundedness principle confirms that $\kappa \coloneqq \sup_{n\in\N} \|B_n\|<\infty$. Thus,
\[
    \norm{B_n x_n - Bx}_{\mathcal{Y}}\leq     \norm{B_n x_n - B_nx}_{\mathcal{Y}}+    \norm{B_n x - Bx}_{\mathcal{Y}} \leq \kappa \norm{x_n -x}_{\mathcal{X}}+\norm{B_n x - Bx}_{\mathcal{Y}} \to 0
\]
as $n\to\infty$.
\end{remark}

\begin{lemma}[{{\cite[Lemma 14.4.6]{SeTrWa22} -- $\div$-$\curl$ lemma beginner's version}}]\label{lem:dclbeg} Let $\mathcal{H}$ be a Hilbert space and $(q_n)_{n\in\N}$, $(r_n)_{n\in\N}$ weakly convergent in $\mathcal{H}$. Furthermore, assume that $\mathcal{X}\subseteq \mathcal{H}$ is a closed subspace and consider $\iota\colon \mathcal{X}\hookrightarrow \mathcal{H}$. If $q_n\in \mathcal{X}$ for $n\in\N$ and $\iota^* r_n\to \iota^*\wlim_{n\to\infty} r_n$ in $\mathcal{X}$, then
\[
   \lim_{n\to\infty}\langle q_n, r_n\rangle_{\mathcal{H}} = \langle \wlim_{n\to\infty} q_n, \wlim_{n\to\infty} r_n\rangle_{\mathcal{H}}.
\]
\end{lemma}
\begin{proof}
Since $q_n\in \mathcal{X}$ for all $n\in \N$, we can write
\[
  \langle q_n, r_n\rangle_{\mathcal{H}} =  \langle \iota \iota^* q_n, r_n\rangle_{\mathcal{H}} = \langle  \iota^* q_n,\iota^* r_n\rangle_{\mathcal{X}}
\]
for $n\in\N$.
The claim thus follows from \Cref{rem:fafact} applied to $B_n \coloneqq  \langle \iota^* q_n ,\cdot \rangle_{\mathcal{X}} \in \Lb(X,\K)$ and $x_n \coloneqq \iota^*r_n \in \mathcal{X}$, $n\in\N$. 
\end{proof}
The next result can be shown with standard approximation techniques using convolutions. We present the result without a proof.
\begin{lemma}\label{lem:philoc} Let $\Omega\subseteq \R^d$ be open and bounded. Let $\phi\in  \Cc(\Omega)$ and $v\in \sobH^1(\Omega)$. Then, $\phi v \in \cH^1(\Omega)$. Moreover, the linear mapping $\sobH^1(\Omega)\ni u \mapsto \phi u \in \cH^1(\Omega)$ is bounded.
\end{lemma}
\begin{theorem}[$\div$-$\curl$ lemma]\label{thm:dcl0} Let $\Omega\subseteq \R^d$ be open and bounded. Let $(q_n)_{n\in\N}$, $(r_n)_{n\in\N}$ in $\Lp{2}(\Omega)^d$ be weakly convergent to $q\coloneqq \wlim_{n\to\infty} q_n$ and $r\coloneqq \wlim_{n\to\infty} r_n$ respectively. Denoting  $\iota_0 \colon\mathfrak{g}_0(\Omega)\hookrightarrow \Lp{2}(\Omega)^d$, we further assume
\[
    \iota_{0}^* r_n\to \iota_0^* r\text{ in }\mathfrak{g}_0(\Omega).
\] 
\begin{enumerate}
\item\label{thm:dcl01} If $q_n\in \mathfrak{g}_0(\Omega)$ for $n\in \N$, then
\begin{equation}\label{eq:conv}
    \forall \phi\in \Cc(\Omega):\int_\Omega \langle r_n(x),q_n(x)\rangle_{\K^d} \phi(x) dx \to     \int_\Omega \langle r(x),q(x)\rangle_{\K^d} \phi(x) dx\text{.}
\end{equation}
\item\label{thm:dcl02} If $\Omega$ has continuous boundary and $q_n\in {\ran}(\grad)$ for $n\in \N$, then~\labelcref{eq:conv} holds as well.
\end{enumerate}
\end{theorem}
\begin{proof}
\labelcref{thm:dcl01}: \Cref{prop:Poincare} shows that $\cgrad$ is one-to-one and that its inverse $\cgrad^{-1} \colon \mathfrak{g}_0 (\Omega)\to H_0^1(\Omega)$ is bounded. Hence, defining $v_n\coloneqq \cgrad^{-1}q_n \in \cH^1(\Omega)$ for $n\in\N$, applying~\labelcref{exer:2.15a} from \Cref{exer:2.15}, and using the fact that closed subspaces are also weakly closed, we obtain $v_n \rightharpoonup v\coloneqq \cgrad^{-1}q$ in $\cH^1(\Omega)$. By \Cref{lem:philoc} and~\labelcref{exer:2.15a} from \Cref{exer:2.15}, $\cgrad (\phi v_n)\rightharpoonup\cgrad (\phi v)$ in $\Lp{2}(\Omega)^d$ for all $\phi\in \Cc(\Omega)$. Moreover, by \Cref{thm:RKT}, $v_n\to v \in \Lp{2}(\Omega)$. Thus, we can
apply~\labelcref{exer:2.15b} from \Cref{exer:2.15} and \Cref{lem:dclbeg} to compute
 \begin{align*}
 \int_\Omega \langle r_n(x),q_n(x)\rangle_{\K^d} \phi(x) d x & = \langle r_n, \phi q_n\rangle_{\Lp{2}(\Omega)^d} \\
 & = \langle r_n, \phi \cgrad v_n\rangle_{\Lp{2}(\Omega)^d} \\
 & = \langle r_n, \cgrad (\phi v_n)\rangle_{\Lp{2}(\Omega)^d} -    \langle r_n,  v_n\cgrad \phi\rangle_{\Lp{2}(\Omega)^d} \\
 & \to \langle r, \cgrad (\phi v)\rangle_{\Lp{2}(\Omega)^d} -  \langle r, v\cgrad \phi \rangle_{\Lp{2}(\Omega)^d}\\ 
 &=  \langle r, \phi q\rangle_{\Lp{2}(\Omega)^d} = \int_\Omega \langle r(x),q(x)\rangle_{\K^d} \phi(x) d x
 \end{align*}
 for all $\phi\in \Cc(\Omega)$.
 
\labelcref{thm:dcl02}: Similar to before, we set, $w_n \coloneqq (\grad\restriction_{\ker(\grad)^\perp})^{-1}q_n \in \sobH^1(\Omega)\cap \ker(\grad)^\perp$.
From \Cref{thm:RKT} (${\ran}(\grad)$ is closed), \Cref{exer:1.375} ($\sobH^1(\Omega)\cap \ker(\grad)^\perp$ is closed), and the open mapping theorem, we obtain bounbedness of $(\grad\restriction_{\ker(\grad)^\perp})^{-1}$ and, therefore,
$w_n\rightharpoonup w\coloneqq (\grad\restriction_{\ker(\grad)^\perp})^{-1} q \in \sobH^1(\Omega)\cap\ker(\grad)^\perp$, and, by \Cref{thm:RKT}, $w_n\to w$ in $\Lp{2}(\Omega)$.  Let $\eta\in \Cc(\Omega)$ with $\eta=1$ on $\supp \phi$. Then, by \Cref{lem:philoc}, $\eta w_n \in \cH^1(\Omega)$ with $\tilde{q}_n \coloneqq  \eta \grad w_n +  w_n\grad \eta= \cgrad (\eta w_n) \in \mathfrak{g}_0(\Omega)$ for $n\in \N$. Again by \Cref{lem:philoc} and~\labelcref{exer:2.15a} from \Cref{exer:2.15}, $\tilde{q}_n\rightharpoonup\cgrad (\eta w)= \eta \grad w +  w\grad \eta \eqqcolon \tilde{q}$ in $\Lp{2}(\Omega)^d$.
 Thus, from~\labelcref{exer:2.15b} from \Cref{exer:2.15} and from~\labelcref{thm:dcl01}, we infer
 \begin{align*}
   \langle r_n, \phi q_n\rangle_{\Lp{2}(\Omega)^d} & =   \langle r_n, \phi (q_n-\tilde{q}_n)\rangle_{\Lp{2}(\Omega)^d}+  \langle r_n, \phi \tilde{q}_n\rangle_{\Lp{2}(\Omega)^d} \\
   & =  \langle r_n, \phi (\grad w_n-\eta \grad w_n -  w_n\cgrad \eta)\rangle_{\Lp{2}(\Omega)^d} + \langle r_n, \phi \tilde{q}_n \rangle_{\Lp{2}(\Omega)^d} \\
   & = -\langle r_n, \phi w_n\cgrad(\eta) \rangle_{\Lp{2}(\Omega)^d}+ \langle r_n, \phi \tilde{q}_n \rangle_{\Lp{2}(\Omega)^d}\\
    &\to-\langle r, \phi w\cgrad(\eta) \rangle_{\Lp{2}(\Omega)^d}+ \langle r, \phi \tilde{q} \rangle_{\Lp{2}(\Omega)^d}\\
 & =\langle r, \phi q \rangle_{\Lp{2}(\Omega)^d}\text{.}\qedhere
 \end{align*}
\end{proof}
Instrumental for our aim is the following application of the $\div$-$\curl$ lemma found by Tartar.

\begin{corollary}\label{cor:dcl} Let $\Omega\subseteq \R^d$ be open and  bounded and $\iota_0 \colon\mathfrak{g}_0(\Omega)\hookrightarrow \Lp{2}(\Omega)^d$. Let $(r_n)_{n\in\N}$ in $\Lp{2}(\Omega)^d$ be weakly convergent to some $r\in\Lp{2}(\Omega)^d$ with $\iota_0^* r_n\to \iota_0^* r$ as $n\to\infty$. Next, assume that $(f_n)_{n\in\N}$ in $\Lp{\infty}(\Omega^{(1)})$ converges to some $f\in 
\Lp{\infty}(\Omega^{(1)})$ with respect to the weak*-topology $\sigma(\Lp{\infty}(\Omega^{(1)}),\Lp{1}(\Omega^{(1)}))$. Then,
\[
\forall \phi\in C_c^\infty(\Omega): \int_\Omega r^{(1)}_n(x)f_n(x^{(1)}) \phi(x) \dd x \to     \int_\Omega r^{(1)}(x)f(x^{(1)}) \phi(x) \dd x\text{.}
 \]
\end{corollary}
\begin{proof} Arguing with a locally finite partition of unity, we can, without loss of generality, assume $\Omega=\prod_{j=1}^d (x_0^{(j)}, x_0^{(j)}+\varepsilon^{(j)})$ for some $\varepsilon \in \R_{>0}^d$ and some $x_0 \in \R^d$. 
The sequence $(f_n)_{n\in\N}$ is uniformly bounded in $\Lp{\infty}(\Omega^{(1)})$ by the uniform boundedness principle. Next, for $n\in\N$, $v_n$ given by
\[
     v_n (x) \coloneqq \int_{x_0^{(1)}}^{x^{(1)}} f_n(s) \dd s
\] belongs to $\sobH^1(\Omega)$ (cf.~\Cref{lem:rand}). Moreover, by \Cref{exer:2.new},
\[
  \ran(\grad) \ni q_n \coloneqq \grad v_n = \begin{pmatrix} f_n, 0, \cdots, 0\end{pmatrix}\rightharpoonup \begin{pmatrix} f, 0, \cdots, 0\end{pmatrix} \in \Lp{2}(\Omega)^d.
\]
Hence, as $
\Omega=\prod_{j=1}^d (x_0^{(j)}, x_0^{(j)}+\varepsilon^{(j)})$ has continuous boundary, the statement follows by \Cref{thm:dcl0}.
\end{proof}

Let $\Omega\subseteq \R^d$ be open.
For the next and future results, we extend the weak partial derivatives $\partial_j$ from \Cref{ex:grad0} to mappings on $\Lp{2}(\Omega)$, via
\begin{equation}\label{eq:definDistDivcon}
    \partial_ j\colon\begin{cases}
\hfill \Lp{2}(\Omega)&\to\quad \sobH^{-1}(\Omega)\\
 \hfill q &\mapsto\quad (\cH^1(\Omega)\ni v\mapsto -\langle q,\partial_j v\rangle_{\Lp{2}(\Omega)})
 \end{cases}
\end{equation}
for $j\in\{1,\dots,d\}$.
It is not difficult to see that this definition is well-defined and is consistent with the one from \Cref{ex:grad0} (\Cref{exer:2.175}). With this notion, we define the following (distributional) version of the divergence
\begin{equation*}
\divcon\colon\begin{cases}
\hfill \Lp{2}(\Omega)^d&\to\quad \sobH^{-1}(\Omega)\\
 \hfill r &\mapsto\quad \sum_{j=1}^d \partial_j r^{(j)}
 \end{cases}\text{.}
\end{equation*}
We can readily show that this is a continuous mapping satisfying 
\begin{equation}\label{eq:definDistDivcon2}
(\divcon r)(v)=-\langle r,\cgrad v\rangle_{\Lp{2}(\Omega)^d}
\end{equation}
 for $r\in\Lp{2}(\Omega)^d$ and
$v\in \cH^1(\Omega)$ (\Cref{exer:2.175}).
\begin{remark}[Strong formulation of variational problems vs weak formulation]\label{rem:StrVarFormVsWeakForm}
Let $\Omega\subseteq \R^d$ open and bounded, $0<\alpha\leq\beta$, $a \in M(\alpha,\beta;\Omega)$, $u\in\cH^1(\Omega)$,
and $f\in\sobH^{-1}(\Omega)$.
We have (\Cref{exer:2.175})
\[
   -\divcon  a \grad_0 u = f \iff [\forall \phi\in C_c^\infty(\Omega)\colon \langle a\grad_0 u ,\grad_0 \phi\rangle_{\Lp{2}(\Omega)^d} = f(\phi)\text{.}]\qedhere
\]
\end{remark}
\begin{theorem}[divergence test]\label{thm:divtest} 
Let $\Omega\subseteq \R^d$ be open and bounded.
Furthermore, let $(r_n)_{n\in\N}$ in $\Lp{2}(\Omega)^d$ be weakly convergent to some $r\in \Lp{2}(\Omega)^d$. Then, the following conditions are equivalent:
\begin{enumerate}
\item $\iota_0^*r_n \to \iota_0^*r$, where $\iota_0\colon \mathfrak{g}_0(\Omega) \hookrightarrow \Lp{2}(\Omega)^d$.
\item $\{\divcon r_n: n\in \N\}\subseteq \sobH^{-1}(\Omega)$ is relatively compact.
\end{enumerate}
\end{theorem}
\begin{proof}
Let $q \in \Lp{2}(\Omega)^d$. Then, for $v\in \cH^1(\Omega)$, we compute
\begin{equation*}
 ( \grad_0^\diamond(q))(v) = \langle q, \grad_0 v\rangle_{\Lp{2}(\Omega)^d} = -(\divcon q)(v).
\end{equation*} 
Hence, $\grad_0^\diamond = -\divcon$. By \Cref{rem:topiso}, $\divcon \iota_0\colon \mathfrak{g}_0 (\Omega)\to H^{-1}(\Omega)$ is a topological isomorphism.  
Thus, $\iota_0^* r_n \to  \iota_0^*r$ in $\mathfrak{g}_0(\Omega)$, if and only if 
\[
\divcon r_n = \divcon \iota_0 \iota_0^*r_n \to \divcon \iota_0 \iota_0^*r = \divcon r
\]
 in $\sobH^{-1}(\Omega)$.
\end{proof}

\section{Examples for $\Htopo$-convergence}

We begin this section by proving \Cref{thm:hcon2d-L2}.

\begin{proof}[Proof of \Cref{thm:hcon2d-L2}] At first we show that~\labelcref{thm:hcon2d-L22} is sufficient for~\labelcref{thm:hcon2d-L21}. Thus, assume that $\alpha_n^{-1} \to \alpha_{\mathrm{h}}^{-1}$ and $\alpha_n\to \alpha_{\mathrm{m}}$ in $\sigma(\Lp{\infty}(\Omega^{(1)}),\Lp{1}(\Omega^{(1)}))$. Then, for $f\in \sobH^{-1}(\Omega)$, consider $(u_n)_{n\in\N}$ in $\cH^1(\Omega)$ be given as the unique solution of
\[
    -\divcon a_n \cgrad u_n = f\text{.}
\] 
Next we use \Cref{prop:subseq}. By \Cref{exer:2.25}, we may choose a subsequence of $(u_n)_{n\in\N}$ in $\cH^1(\Omega)$ and $(q_n)_{n\in\N} \coloneqq (a_n\grad u_n)_{n\in\N}$ in $\Lp{2}(\Omega)^2$ both weakly convergent to some $u\in\cH^1(\Omega)$ and $q\in \Lp{2}(\Omega)^2$, respectively. As we shall uniquely identify $u$ and $q$ in the following, \Cref{prop:subseq} will imply weak convergence of the original sequence. In order to reduce clutter in the notation as much as possible we already did and shall re-use the indices $n$ to denote any subsequences. Note that $-\divcon a_n \cgrad u_n=f$ and, thus, by \Cref{thm:divtest}, $\iota_0^*q_n \to \iota_0^* q $. By \Cref{cor:dcl}, we see that
\begin{multline*}
   \langle\phi,\partial_1 u\rangle_{\Lp{2}(\Omega)} =\lim_{n\to\infty}  \langle\phi,\partial_1 u_n\rangle_{\Lp{2}(\Omega)} = \lim_{n\to\infty}\langle \phi, \alpha_n^{-1}\alpha_n  \partial_1 u_n\rangle_{\Lp{2}(\Omega)}\\
 =\lim_{n\to\infty} \langle\phi ,\alpha_n^{-1} q^{(1)}_n\rangle_{\Lp{2}(\Omega)} =\langle\phi, \alpha_{\mathrm{h}}^{-1} q^{(1)}\rangle_{\Lp{2}(\Omega)}
\end{multline*}
for all $\phi\in \Cc (\Omega)$ and, hence, $\alpha_{\mathrm{h}}\partial_1 u = q^{(1)}$.
Moreover, consider $r_n \coloneqq (\partial_2 u_n, -\partial_1 u_n)\weakto (\partial_2 u, -\partial_1 u)\eqqcolon r$. By \Cref{thm:divtest}, we deduce from $\divcon r_n =0$ that \Cref{cor:dcl} is once again applicable. Therefore, 
\begin{multline*}
   \langle\phi, q^{(2)}\rangle_{\Lp{2}(\Omega)}=\lim_{n\to\infty}\langle\phi, q_n^{(2)}\rangle_{\Lp{2}(\Omega)} =\lim_{n\to\infty} \langle\phi,\alpha_n \partial_2 u_n\rangle_{\Lp{2}(\Omega)}\\
=\lim_{n\to\infty}\langle\phi, \alpha_n r_n^{(1)}\rangle_{\Lp{2}(\Omega)} = \langle\phi,\alpha_{\mathrm{m}}r^{(1)}\rangle_{\Lp{2}(\Omega)}=\langle\phi,\alpha_{\mathrm{m}}\partial_2 u\rangle_{\Lp{2}(\Omega)}
\end{multline*}
for all $\phi\in \Cc (\Omega)$ and, hence, $q^{(2)}=\alpha_{\mathrm{m}}\partial_2 u$.
Thus, for $v\in \cH^1 (\Omega)$, we deduce
\begin{multline*}
  \langle \diag (\alpha_{\mathrm{h}}, \alpha_{\mathrm{m}})\grad_0 u,\grad_0v\rangle_{\Lp{2}(\Omega)^2}=
 \langle q,\cgrad v\rangle_{\Lp{2}(\Omega)^2}\\
=\lim_{n\to\infty}  \langle q_n,\cgrad v\rangle_{\Lp{2}(\Omega)^2}=\lim_{n\to\infty}  \langle a_n\cgrad u_n,\cgrad v\rangle_{\Lp{2}(\Omega)^2} = f(v)\text{,}
\end{multline*}
i.e., $-\divcon a\cgrad u= f$. Since this $u$ is unique, we obtain weak convergence of  the original sequence $(u_n)_{n\in\N}$ and, thus, of $(q_n)_{n\in\N}$ to $u$ and $\diag (\alpha_{\mathrm{h}}, \alpha_{\mathrm{m}}) \cgrad u$, respectively.

For the opposite direction,  by (weak*-sequential) compactness, we may choose a subsequence of $(\alpha_n)_{n\in\N}$ such that $\alpha_n^{-1}\to \tilde{\alpha}_{\mathrm{h}}^{-1}\in \Lp{\infty}(\Omega^{(1)})$ and $\alpha_n \to \tilde{\alpha}_{\mathrm{m}}\in\Lp{\infty}(\Omega^{(1)})$ in $\sigma(\Lp{\infty}(\Omega^{(1)}),\Lp{1}(\Omega^{(1)}))$ as $n\to\infty$. Next, let $u_0 \in \cH^1(\Omega)$ and define $f\colon \cH^1(\Omega) \to \C$ via
\[
f(v) \coloneqq \langle \diag(\alpha_{\mathrm{h}},\alpha_{\mathrm{m}})\grad u_0,\grad v\rangle_{\Lp{2}(\Omega)^2}\text{.}
\]
 For $n\in\N$, define $u_n\in \cH^1(\Omega)$ by
\[
   -\divcon  a_n \cgrad u_n  = f\text{.}
\]
Then, since $a_n\stackrel{\Htopo}{ \to} \diag(\alpha_{\mathrm{h}},\alpha_{\mathrm{m}})$, we infer that
$u_n\rightharpoonup u_0$ in $\cH^1(\Omega)$, and
\[
    a_n \cgrad u_n \rightharpoonup \diag(\alpha_{\mathrm{h}},\alpha_{\mathrm{m}})\cgrad u_0 
\]
in $\Lp{2}(\Omega)^2$.
 Furthermore, by \labelcref{thm:hcon2d-L22}$\Rightarrow$\labelcref{thm:hcon2d-L21}, $u_n\rightharpoonup u$ in $\cH^1(\Omega)$, and
\[
   a_n\cgrad u_n \to \diag(\tilde\alpha_{\textnormal{h}},\tilde\alpha_{\textnormal{m}})\cgrad u
\]
in $\Lp{2}(\Omega)^2$.
Hence, $u=u_0$ and 
\[
    \diag(\alpha_{\mathrm{h}},\alpha_{\mathrm{m}})\grad u_0 = \diag(\tilde\alpha_{\textnormal{h}},\tilde\alpha_{\textnormal{m}})\grad u_0.
\]
Since $u_0$ was arbitrary it follows that $\diag(\tilde\alpha_{\textnormal{h}},\tilde\alpha_{\textnormal{m}})=\diag(\alpha_{\textnormal{h}},\alpha_{\textnormal{m}})$.
\end{proof}

With \Cref{thm:hcon2d-L2} proven, we see that even for scalar coefficients, the operator-theoretic description of $\Htopo$-convergence becomes somewhat more difficult, if the dimension of the underlying domain changes from $1$ to $2$. In fact, for periodic coefficients, an $\Htopo$-convergence result is known as well. However, the formula for the limit becomes less explicit:

Let $a\colon \R^d \to \C^{d\times d}$ be bounded and measurable. Assume
there is $c>0$ such that $\Re a(\cdot)\geq c$, and assume $a$ to be $Y\coloneqq (0,1)^d$-\textbf{periodic}, that is, for all $k\in \Z^d$, $a(k+\cdot)=a$. For $n\in \N$, we denote $a_n\coloneqq a(n\cdot)$. Then we find $0<\alpha\leq\beta$ such that $a_n \in M(\alpha,\beta;\Omega)$ for all $\Omega\subseteq \R^d$ open and bounded and $n\in \N$. In order to state the $\Htopo$-convergence statement for $(a_n)_{n\in\N}$, we need to recall the following result. For this, in turn, we introduce for $w \in (0,1)^{d-1}$ and $j\in \{1,\ldots,d\}$
\[
   \check{w}_j \coloneqq (w^{(1)},\ldots,w^{(j-1)}, 0, w^{(j)},\ldots, w^{(d-1)}), \text{ and } \hat{w}_j \coloneqq    \check{w}_j + (\delta_{ij})_{i\in \{1,\ldots,d\}}\text{,}
\]
as well as
 \[
\mathrm{C}_\#^1(Y) \coloneqq \{u\in \mathrm{C}^1(\overline{Y}): u( \check{w}_j) = u (\hat{w}_j), w \in (0,1)^{d-1}, j\in \{1,\ldots,d\} \}\text{.}
\]
Furthermore, similarly to $\cgrad$ on $\cH^1(\Omega)$, we can weakly extend $\grad\restriction_{C_\#^1(Y)}$ to $\grad_\#$ on the closure of $C_\#^1(Y)$ in $\sobH^1(Y)$. Then, the result reads as follows:
\begin{theorem}\label{thm:compcoeff2} Let $\xi\in \C^d$. Then, there exists a unique $v_\xi\in \Lp{2}(Y)^d$ such that
\[
   v_\xi - \xi \in \ran(\grad_\#)\text{ and } av_\xi \in \ker(\grad_\#^\diamond)\text{.}
\]
\end{theorem}
We define $a_{\textnormal{hom}}$ by
\begin{equation}\label{eq:homcoef}
   \C^d \ni \xi \mapsto \int_Y av_\xi \in \C^d,
\end{equation}where $v_\xi \in \Lp{2}(Y)^d$ is constructed in \Cref{thm:compcoeff2}.

\begin{theorem}\label{thm:Hconper} Let $\Omega\subseteq \R^d$ be open and bounded. Then, the sequence $(a_n)_{n\in\N}$ $\Htopo$-converges to $a_{\textnormal{hom}}$.
\end{theorem}

Even though being within the scope of the techniques of the present course, we shall not provide a proof of these classical results here. However, we will move on and introduce the Schur topology in the next section to open up a wider viewpoint on homogenisation.

\section{The Schur topology and its relation to $\Htopo$-convergence}

We start off with the definition of the Schur topology. For this, let $\mathcal{H}$ be a Hilbert space and $\mathcal{H}_0\subseteq \mathcal{H}$ be a closed subspace, and define $\mathcal{H}_1\coloneqq \mathcal{H}_0^\bot$. For $a\in \Lb(\mathcal{H})$ we define $a_{jk}\coloneqq \iota_j^*a\iota_k$ for all $j,k\in \{0,1\}$ with $\iota_{0/1}\colon \mathcal{H}_{0/1}\hookrightarrow \mathcal{H}$. Next, we define
\[
    \mathcal{M}(\mathcal{H}_0,\mathcal{H}_1)\coloneqq \{a\in L(\mathcal{H}); a_{00}\text{ and } a \text{ continuously invertible}\}\text{.}
\]
Note that for $\Omega\subseteq\R^d$ open and bounded and $0<\alpha\leq\beta$, we have, by interpreting any $\Lp{\infty}$-function as bounded linear operator in the corresponding $\Lp{2}$-space,
\[
    M(\alpha,\beta;\Omega)\subseteq  \mathcal{M}(\mathcal{H}_0,\mathcal{H}_1)
\]for all closed subspaces $\mathcal{H}_0 \subseteq \Lp{2}(\Omega)^d$ by \Cref{lem:pda}.

The \emph{Schur topology}, $\tau(\mathcal{H}_0,\mathcal{H}_1)$, on $\mathcal{M}(\mathcal{H}_0,\mathcal{H}_1)$ is the initial topology induced by the mappings
\begin{align*}
   a& \mapsto a_{00}^{-1}   \in \Lb^{\mathrm{w}}(\mathcal{H}_0)\\
   a & \mapsto a_{00}^{-1}a_{01} \in \Lb^{\mathrm{w}}(\mathcal{H}_1,\mathcal{H}_0)\\
   a& \mapsto a_{10}a_{00}^{-1} \in \Lb^{\mathrm{w}}(\mathcal{H}_0,\mathcal{H}_1) \\
   a&\mapsto a_{11}-a_{10}a_{00}^{-1}a_{01} \in \Lb^{\mathrm{w}}(\mathcal{H}_1)\text{,}
\end{align*}where the superscript $\mathrm{w}$ means that the corresponding spaces of bounded linear operators are endowed with the weak operator topology.

\begin{example}\label{ex:trivialcases}
If $\mathcal{H}_0=\mathcal{H}$, then $(a_n)_{n\in\N} \to a$ in $\tau(\mathcal{H},\{0\})$ if and only if $(a_n^{-1})_{n\in\N}\to a^{-1}$ in the weak operator topology.

If $\mathcal{H}_1=\mathcal{H}$, then $(a_n)_{n\in\N} \to a$ in $\tau(\{0\},\mathcal{H})$ if and only if $(a_n)_{n\in\N}\to a$ in the weak operator topology.
\end{example}

The fundamental relationship between $\Htopo$-convergence and convergence with respect to the Schur topology reads as follows.

\begin{theorem}\label{thm:HnlH} Let $\Omega\subseteq \R^d$ be open and bounded, $0<\alpha\leq\beta$, and $(a_n)_{n\in\N}, a$ in $M(\alpha,\beta;\Omega)$. Then, the following conditions are equivalent:
\begin{enumerate}
  \item\label{thm:HnlH1} $(a_n)_{n\in\N}$ $\Htopo$-converges to $a$;
  \item\label{thm:HnlH2} $(a_n)_{n\in\N}$ $\tau(\mathfrak{g}_0(\Omega),\mathfrak{g}_0(\Omega)^\perp)$-converges to $a$.
\end{enumerate}
\end{theorem}

For the proof of this theorem, we need some preliminaries. An inspection of  the proof of \Cref{thm:HnlH} below shows that, actually, condition~\labelcref{thm:HnlH2} is clearly stronger than~\labelcref{thm:HnlH1}. So, the main part will be showing that~\labelcref{thm:HnlH1} also implies~\labelcref{thm:HnlH2}. For this, we need the following refinement of $\Htopo$-convergence. 

\begin{theorem}\label{thm:Hcon2} Let $\Omega\subseteq \R^d$ be open and bounded, $0<\alpha\leq
\beta$, $(a_n)_{n\in\N}$, $a$ in $M(\alpha,\beta;\Omega)$. Assume that $(a_n^*)_{n\in\N}$ $\Htopo$-converges to $a^*\in M(\alpha,\beta;\Omega)$. Then, for all $z\in \Lp{2}(\Omega)^d, f\in \sobH^{-1}(\Omega)$ and $u_n\in \cH^1(\Omega)$ satisfying
\[
   -\divcon a_n (\cgrad u_n+z) = f\text{,}
\]
we have $u_n\rightharpoonup u\in \cH^1(\Omega)$ and $a_n(\cgrad u_n +z)\rightharpoonup a(\cgrad u+z)\in \Lp{2}(\Omega)^d$, where $u\in \cH^1(\Omega)$ satisfies
\[
   -\divcon a (\cgrad u+z) = f\text{.}
\]
\end{theorem}
\begin{proof}
  First of all, note that $(u_n)_{n\in\N}$ is well-defined and bounded in $\cH^1(\Omega)$. Also $u\in \cH^1(\Omega)$ given as the solution of the mentioned variational problem exists and is uniquely determined. Define the bounded $\Lp{2}(\Omega)^d$-sequence $q_n\coloneqq a_n(\cgrad u_n+z)$ for $n\in\N$ and choose subsequences (not relabeled, cf.~\Cref{prop:subseq}) such that $q_n\rightharpoonup  q\in \Lp{2}(\Omega)^d$ and $u_n\rightharpoonup w\in H_0^1(\Omega)$.
  
  Let $v\in \cH^1(\Omega)$ and define $v_n \in \cH^1(\Omega)$ as the solution of
  \[
      -\divcon a_n^*\cgrad v_n = -\divcon a^*\cgrad v\text{.}
  \]
  Then, as $(a_n^*)_{n\in\N}$ $\Htopo$-converges to $a^*$, $v_n \rightharpoonup v$ in $\cH^1(\Omega)$ and $a_n^*\cgrad v_n \rightharpoonup a^*\cgrad v$ in $\Lp{2}(\Omega)^d$ as $n\to\infty$. Next, we consider 
  \begin{multline*}
     \langle a_n (\cgrad u_n + z), \cgrad v_n\rangle_{\K^d}=     \langle  (\cgrad u_n + z), a_n^*\cgrad v_n\rangle_{\K^d} \\ = \langle \cgrad u_n, a_n^*\cgrad v_n\rangle_{\K^d} + \langle z, a_n^*\cgrad v_n\rangle_{\K^d}.
  \end{multline*}
  The divergence test, \Cref{thm:divtest}, shows that both $\iota_0^* q_n\to \iota_0^* q$ and $\iota_0^* a_n^*\grad v_n \to \iota_0^* a^*\grad v$. Hence, applying \Cref{thm:dcl0}, we deduce for all $\phi\in \Cc(\Omega)$
  \begin{multline*}
     \int_{\Omega} \langle q, \cgrad v\rangle_{\K^d} \phi = \int_{\Omega} \langle \cgrad w, a^*\grad v\rangle_{\K^d}\phi + \int_{\Omega} \langle z, a^*\cgrad v\rangle_{\K^d} \phi\\ = \int_{\Omega} \langle a(\cgrad w+z),\grad v\rangle_{\K^d} \phi.
  \end{multline*}
  Since both $v$ and $\phi$ were arbitrary, we infer that
  \[
     q = a(\cgrad w+z).
  \]
  In particular, for all $\phi\in \cH^1(\Omega)$,
  \[
     \langle a(\cgrad w+z), \cgrad \phi\rangle_{\Lp{2}(\Omega)^d} = \lim_{n\to\infty} \langle q_n, \cgrad \phi\rangle_{\Lp{2}(\Omega)^d}= f(\phi)\text{,}
  \]
  hence, $w=u$.
\end{proof}
\begin{corollary}\label{cor:Hconadj} Let $\Omega\subseteq \R^d$ be open and bounded, $0<\alpha\leq\beta$, $(a_n)_{\in\N}, a$ in $M(\alpha,\beta;\Omega)$. Then,
\[
(a_n)_{n\in\N} \text{ $\Htopo$-converges to } a \iff (a_n^*)_{n\in\N} \text{ $\Htopo$-converges to } a^*.
\]\end{corollary}
\begin{proof}
Since $(a_n^{**})_{n\in\N}=(a_n)_{n\in\N}$, it suffices to prove one implication only. Putting $z=0$ in \Cref{thm:Hcon2}, we deduce that $(a_n^*)_{n\in\N}$ $\Htopo$-converging to $a^*$ is sufficient for $(a_n)_{n\in\N}$ $\Htopo$-converging to  $a$.
\end{proof}

\begin{proposition}\label{prop:equivprobl2} Let $\Omega\subseteq \R^d$ be open and bounded, $0<\alpha\leq
\beta$, $a \in M(\alpha,\beta; \Omega)$, and $z\in \Lp{2}(\Omega)^d$. Consider the following problems:
\begin{enumerate}
 \item\label{prop:equivprobl21} Find $u\in \cH^1(\Omega)$ such that 
 \[
     \forall \phi \in \cH^1(\Omega): \langle a (\cgrad u + z),\cgrad \phi\rangle_{\Lp{2}(\Omega)^d} =0.
 \]
 \item\label{prop:equivprobl22} Find $p\in \mathfrak{g}_0(\Omega)^\perp$ such that 
 \[
   \forall q \in \mathfrak{g}_0(\Omega)^\perp : \langle a^{-1} p, q\rangle_{\Lp{2}(\Omega)^d} = \langle z,q\rangle_{\Lp{2}(\Omega)^d}.
 \]
\end{enumerate}
Then, both~\labelcref{prop:equivprobl21} and~\labelcref{prop:equivprobl22} admit uniquely determined solutions. If $u$ and $p$ are the respective solutions, then $p = a(\cgrad u + z)$.
\end{proposition}
\begin{proof}
Unique existence of solutions is a standard application of \Cref{cor:TW14}.
Next, let $u\in \cH^1(\Omega)$ be a solution for~\labelcref{prop:equivprobl21}. Then $p\coloneqq a(\cgrad u + z) \in \mathfrak{g}_0(\Omega)^\perp$. Moreover, for $q\in \mathfrak{g}_0(\Omega)^\perp$ we compute
\[
   \langle a^{-1} p, q\rangle_{\Lp{2}(\Omega)^d}=   \langle a^{-1} (a(\grad u+z)), q\rangle_{\Lp{2}(\Omega)^d} =  \langle(\grad u+z), q\rangle_{\Lp{2}(\Omega)^d} = \langle z,q\rangle_{\Lp{2}(\Omega)^d}\text{.}
\] Uniqueness of~\labelcref{prop:equivprobl22} implies that $a(\cgrad u + z)$ is the only solution for~\labelcref{prop:equivprobl22}.
\end{proof}

\begin{proof}[Proof of \Cref{thm:HnlH}] \labelcref{thm:HnlH2}$\Rightarrow$\labelcref{thm:HnlH1}: Let $f\in \sobH^{-1}(\Omega)$ and $u_n\in \cH^1(\Omega)$ be the solutions of
\[
    \forall \phi\in \cH^1(\Omega):\langle a_n \cgrad u_n,\cgrad \phi\rangle_{\Lp{2}(\Omega)^d} = f(\phi)
\]
for $n\in\N$.
Then, by \Cref{cor:TW14} and \Cref{rem:topiso}, with $\iota_0 \colon \mathfrak{g}_0 (\Omega)\hookrightarrow L_2(\Omega)^d$,
\[
   u_n = (\iota_0^* \cgrad)^{-1} (\iota_0^*a_n \iota_0)^{-1} (\cgrad^\diamond \iota_0)^{-1}f
\]
for $n\in\N$, where both $(\iota_0^* \cgrad)^{-1}$ and $(\cgrad^\diamond \iota_0)^{-1}$ are topological isomorphisms.
Hence, as $(\iota_0^*a_n \iota_0)^{-1} = a_{n,00}^{-1}\to a_{00}^{-1}=(\iota_0^*a \iota_0)^{-1}$ in the weak operator topology by definition of $\tau(\mathfrak{g}_0(\Omega),\mathfrak{g}_0(\Omega)^\bot)$, $u_n\rightharpoonup u \coloneqq  (\iota_0^* \cgrad)^{-1} (\iota_0^*a \iota_0)^{-1} (\cgrad^\diamond \iota_0)^{-1}f \in \cH^1(\Omega)$. Hence, $u$ is the solution of
\[
    \forall \phi\in \cH^1(\Omega):\langle a \cgrad u,\cgrad \phi\rangle_{\Lp{2}(\Omega)^d} = f(\phi)\text{.}
\]
Finally, introducing $\iota_1\colon \mathfrak{g}_0(\Omega)^\bot\hookrightarrow L_2(\Omega)^d$ and using that $\begin{pmatrix} \iota_0 & \iota_1\end{pmatrix}^* = \begin{pmatrix} \iota_0^* \\ \iota_1^* \end{pmatrix}$ is unitary, we compute
\begin{align*}
 a_n\cgrad u_n & =  a_n\iota_0 \iota_0^*\cgrad u_n\\
  & =  a_n\iota_0 \iota_0^* \cgrad (\iota_0^* \cgrad)^{-1} (\iota_0^*a_n \iota_0)^{-1} (\cgrad^\diamond \iota_0)^{-1}f \\
  & =  a_n\iota_0  (\iota_0^*a_n \iota_0)^{-1} (\cgrad^\diamond \iota_0)^{-1}f \\
    & = \begin{pmatrix} \iota_0 & \iota_1\end{pmatrix}\begin{pmatrix} \iota_0^* \\ \iota_1^* \end{pmatrix} a_n \begin{pmatrix} \iota_0 & \iota_1\end{pmatrix}\begin{pmatrix} \iota_0^* \\ \iota_1^* \end{pmatrix} \iota_0  a_{n,00}^{-1} (\cgrad^\diamond \iota_0)^{-1}f\\
    & = \begin{pmatrix} \iota_0 & \iota_1\end{pmatrix}\begin{pmatrix} a_{n,00} & a_{n,01} \\ a_{n,10} & a_{n,11} \end{pmatrix}\begin{pmatrix} \iota_0^*\iota_0 \\ \iota_1^*\iota_0 \end{pmatrix}   a_{n,00}^{-1} (\cgrad^\diamond \iota_0)^{-1}f\\
        & = \begin{pmatrix} \iota_0 & \iota_1\end{pmatrix}\begin{pmatrix} a_{n,00} & a_{n,01} \\ a_{n,10} & a_{n,11} \end{pmatrix}\begin{pmatrix}  a_{n,00}^{-1} (\cgrad^\diamond \iota_0)^{-1}f
 \\ 0\end{pmatrix}  \\
 & = \begin{pmatrix} \iota_0 & \iota_1\end{pmatrix}\begin{pmatrix} (\cgrad^\diamond \iota_0)^{-1}f \\ a_{n,10} a_{n,00}^{-1} (\cgrad^\diamond \iota_0)^{-1}f\end{pmatrix} \\& \rightharpoonup \begin{pmatrix} \iota_0 & \iota_1\end{pmatrix}\begin{pmatrix} (\cgrad^\diamond \iota_0)^{-1}f \\ a_{10} a_{00}^{-1} (\cgrad^\diamond \iota_0)^{-1}f\end{pmatrix} = a\cgrad u\text{.}
\end{align*}
Thus, $(a_n)_{n\in\N}$ $\Htopo$-converges to $a$.

\labelcref{thm:HnlH1}$\Rightarrow$\labelcref{thm:HnlH2}: Let $f\in \cH^{-1}(\Omega)$ and $u_n \in H_0^1(\Omega)$ for $n\in\N$ satisfy
\[
   \forall \phi\in H_0^1(\Omega): \langle a_n \cgrad u_n,\cgrad \phi\rangle_{\Lp{2}(\Omega)^d} = f(\phi)\text{,}
\]
and similar for $a$ and $u\in\cH^1(\Omega)$. Then, using the reformulations just carried out above, we deduce
\[
    (\iota_0^* \cgrad)^{-1} (\iota_0^*a_n \iota_0)^{-1} (\cgrad^\diamond \iota_0)^{-1}f= u_n \rightharpoonup u = (\iota_0^* \cgrad)^{-1} (\iota_0^*a \iota_0)^{-1} (\cgrad^\diamond \iota_0)^{-1}f\text{,}
\]
and
\begin{multline*}
\begin{pmatrix} \iota_0 & \iota_1\end{pmatrix}\begin{pmatrix} (\cgrad^\diamond \iota_0)^{-1}f \\ a_{n,10} a_{n,00}^{-1} (\cgrad^\diamond \iota_0)^{-1}f\end{pmatrix} = a_n\cgrad u_n \\ \rightharpoonup  a\cgrad u=
\begin{pmatrix} \iota_0 & \iota_1\end{pmatrix}\begin{pmatrix} (\cgrad^\diamond \iota_0)^{-1}f \\ a_{10} a_{00}^{-1} (\cgrad^\diamond \iota_0)^{-1}f\end{pmatrix}\text{.} 
\end{multline*}
Since $f$ was arbitrary and $ (\iota_0^* \grad_0)^{-1}$, $(\grad_0^\diamond \iota_0)^{-1}$, and $\begin{pmatrix} \iota_0 & \iota_1\end{pmatrix}$ are topological isomorphisms, we infer that $a_{n,00}^{-1}\to a_{00}^{-1}$ and $a_{n,10} a_{n,00}^{-1}\to a_{10} a_{00}^{-1}$ in the weak operator topology.

Before we address the remainder of the proof, we quickly convince ourselves that
\[
    \begin{pmatrix} 1 &0 \\  - a_{10}a_{00}^{-1} & 1 \end{pmatrix}  \begin{pmatrix} a_{00} & a_{01} \\ a_{10} & a_{11} \end{pmatrix}
    \begin{pmatrix} 1 & -a_{00}^{-1}a_{01} \\0 & 1 \end{pmatrix}
     =
\begin{pmatrix} a_{00} & 0 \\ 0 & a_{11}-a_{10}a_{00}^{-1}a_{01} \end{pmatrix}\text{.}      
\]
Since $a$ is invertible, so is $a_S\coloneqq a_{11}-a_{10}a_{00}^{-1}a_{01}$, and we have
\[
   \begin{pmatrix} a_{00} & a_{01} \\ a_{10} & a_{11} \end{pmatrix}^{-1} =
    \begin{pmatrix} a_{00}^{-1} + a_{00}^{-1}a_{01} a_S^{-1} a_{10} a_{00}^{-1} & - a_{00}^{-1}a_{01}a_S^{-1} \\ -a_{S}^{-1}a_{10}a_{00}^{-1} & a_S^{-1}\end{pmatrix}\text{.}
   \]
   In particular, $(a^{-1})_{11}^{-1} = a_{11}-a_{10}a_{00}^{-1}a_{01}$. Similar formulas hold for $a_n$, $n\in \N$.
   
   Next, let $z\in \mathfrak{g}_0(\Omega)^{\bot}$ and $p_n \coloneqq \iota_1(a_{n,11}-a_{n,10}a_{n,00}^{-1}a_{n,01})z \in \mathfrak{g}_0(\Omega)^{\perp}\subseteq \Lp{2}(\Omega)^d$ for $n\in\N$. Then, for all $q\in \mathfrak{g}_0(\Omega)^\perp$,
   \[
        \langle z,q\rangle_{\Lp{2}(\Omega)^d} = \langle (a_n^{-1})_{11} p_n,q\rangle_{\Lp{2}(\Omega)^d} =\langle a_n^{-1}  p_n, q\rangle_{\Lp{2}(\Omega)^d}\text{.}
   \]
   By \Cref{prop:equivprobl2}, $p_n =a_n (\cgrad u_n + z)$, $n\in\N$, where $u_n\in \cH^1(\Omega)$ satisfies
   \[
     \forall \phi \in \cH^1(\Omega)\colon \langle a_n (\cgrad u_n + z),\cgrad \phi\rangle_{\Lp{2}(\Omega)^d} =0\text{.}
 \]
   Since $(a_n)_{n\in\N}$ $\Htopo$-converges to $a$, by \Cref{cor:Hconadj}, $(a_n^*)_{n\in\N}$ $\Htopo$-converges to $a^*$. Thus, by \Cref{thm:Hcon2}, $u_n\rightharpoonup u \in \cH^1(\Omega)$ and $a_n(\cgrad u_n + z)\rightharpoonup a(\cgrad u + z)\eqqcolon p \in \Lp{2}(\Omega)^d$. By \Cref{prop:equivprobl2}, $p$ satisfies
\[
       \forall q\in \mathfrak{g}_0(\Omega)^\perp : \langle z,q\rangle_{\Lp{2}(\Omega)^d} =\langle a^{-1}  p, q\rangle_{\Lp{2}(\Omega)^d} \text{.}
   \]
   In particular, $p = \iota_1(a_{11}-a_{10}a_{00}^{-1}a_{01})z$, and we deduce, by the arbitrariness of $z$, that $(a_{n,11}-a_{n,10}a_{n,00}^{-1}a_{n,01})\to (a_{11}-a_{10}a_{00}^{-1}a_{01})$ in the weak operator topology as $n\to\infty$. Moreover,    \begin{align*}
       a_n^{-1} p_n & =      \begin{pmatrix} \iota_0 & \iota_1\end{pmatrix}\begin{pmatrix} \iota_0^* \\ \iota_1^* \end{pmatrix}  a_n^{-1}\begin{pmatrix} \iota_0 & \iota_1\end{pmatrix}\begin{pmatrix} \iota_0^* \\ \iota_1^* \end{pmatrix} p_n \\
       &=      \begin{pmatrix} \iota_0 & \iota_1\end{pmatrix}    \begin{pmatrix} \ast & - a_{n,00}^{-1}a_{n,01}a_{n,S}^{-1} \\ \ast & a_{n,S}^{-1}\end{pmatrix}
\begin{pmatrix} 0 \\ \iota_1^*p_n \end{pmatrix} \\
  & =  \begin{pmatrix} \iota_0 & \iota_1\end{pmatrix}    \begin{pmatrix} * & - a_{n,00}^{-1}a_{n,01}a_{n,S}^{-1} \\ * & a_S^{-1}\end{pmatrix}
\begin{pmatrix} 0 \\ a_{n,S} z \end{pmatrix} \\
  & =  \begin{pmatrix} \iota_0 & \iota_1\end{pmatrix}    \begin{pmatrix}  - a_{n,00}^{-1}a_{n,01}z \\  z\end{pmatrix} = - \iota_0 a_{n,00}^{-1}a_{n,01}z+\iota_1 z.
   \end{align*}
 Hence,    as $u_n\rightharpoonup u \in \cH^1(\Omega)$,
 \begin{align*}
    - \iota_0 a_{n,00}^{-1}a_{n,01}z & = a_n^{-1}p_n - \iota_1 z \\ & = a_n^{-1}a_n(\grad u_n +z) -z  = \grad u_n \rightharpoonup \grad u = - \iota_0 a_{00}^{-1}a_{01}z,
 \end{align*}
 which again, as $z$ was arbitrary, establishes the remaining convergence.
\end{proof}

Now, that we have proven that $\Htopo$-convergence can be abstractly described by means of convergence in an operator topology (dismantling any assumptions of the coefficients being multiplication operators), homogenisation theorems can be spelled out for abstract operator equations. This is content of our last lecture, where we will show homogenisation results for partial differential equations.

\section{Comments}
A complete definition of continuous boundary and a proof of both~\labelcref{thm:RKT1} and~\labelcref{thm:RKT2} from~\Cref{thm:RKT} can be found in, e.g.,~\cite{ISem18} or~\cite[Chapter~5]{EE87}.
A proof of \Cref{thm:compcoeff2} can, for instance, be found in \cite[Lemma 14.4.1]{SeTrWa22}. For \Cref{cor:dcl} and \Cref{thm:hcon2d-L2},
compare \cite[Chapters~5 and~12]{Ta09}.
The self-adjoint coefficient case of \Cref{thm:Hcon2} and \Cref{prop:equivprobl2} has been provided in \cite{ZKO94}. The proofs provided here are based on \cite{Preprint}. \Cref{thm:HnlH} is based on~\cite[Section~4]{Wa18} and~\cite{Preprint}.
\begin{takehomes}\begin{enumerate}
\item[(a)] The basic principle of showing any kind of convergence in homogenisation is the subsequence principle in conjunction with weak (sequential) compactness.
\item[(b)] The div-curl lemma is a statement allowing to deduce convergence of products of weakly converging sequence. The way spelled out here, it works as a localisation tool.
\item[(c)] Homogenisation in higher dimensions yields different homogenisation formlas. In particular, scalar coefficients are not closed under homogenisation.
\item[(d)] $\Htopo$-convergence is equivalent to convergence in the Schur topology for an appropriate decomposition of the space.
\item[(e)] $\Htopo$-convergence for multiplication operators yields convergence for a different variational problem for free.
\end{enumerate}
\end{takehomes}

\section{Exercises}
\begin{exercise}\label{exer:2.1} Let $\mathcal{H}_0,\mathcal{H}_1$ be Hilbert spaces and $A\colon \dom(A)\subseteq \mathcal{H}_0\to \mathcal{H}_1$ be a closed, linear and densely defined operator. If $\dom(A)\cap \ker(A)^\bot\hookrightarrow \mathcal{H}_0$ is compact, prove that then $\ran(A)\subseteq \mathcal{H}_1$ is closed.
\textrm{Hint:} Argue by contradiction and use \Cref{exer:1.2}.
\end{exercise}

\begin{exercise}\label{exer:2.15}
\begin{enumerate}[label=\textup{(\alph{*})}]
\item\label{exer:2.15a} Let $\mathcal{X},\mathcal{Y}$ be Banach spaces and $B\in\Lb(\mathcal{X},\mathcal{Y})$. Show that $B\colon \mathcal{X}\to \mathcal{Y}$ is
continuous if we endow both $\mathcal{X}$ and $\mathcal{Y}$ with its respective weak topology.
\item\label{exer:2.15b} Let $\mathcal{H}$ be a Hilbert space and consider $(a_n)_{n\in\N}$, $(b_n)_{n\in\N}$, $a$, and $b$ in $\mathcal{H}$. Show that
$\lim_{n\to\infty}a_n =a$ and $\wlim_{n\to\infty}b_n=b$ imply $\langle a_n,b_n \rangle_{\mathcal{H}}\to \langle a,b \rangle_{\mathcal{H}}$ as $n\to\infty$.
\end{enumerate}
\end{exercise}
\begin{exercise}\label{exer:2.new}
Let $\Omega\subseteq \R^d$ be open and  bounded and assume that $(f_n)_{n\in\N}$ in $\Lp{\infty}(\Omega^{(1)})$ converges to some $f\in 
\Lp{\infty}(\Omega^{(1)})$ with respect to the weak*-topology $\sigma(\Lp{\infty}(\Omega^{(1)}),\Lp{1}(\Omega^{(1)}))$.
Show that $\begin{pmatrix} f_n, 0, \cdots, 0\end{pmatrix}\rightharpoonup \begin{pmatrix} f, 0, \cdots, 0\end{pmatrix}$ in $\Lp{2}(\Omega)^d$.
\end{exercise}
\begin{exercise}\label{exer:2.175}
Show that~\labelcref{eq:definDistDivcon} is well-defined and consistent with, i.e., extends the definition from, \Cref{ex:grad0} (using $\Lp{2}(\Omega)\subseteq\sobH^{-1}(\Omega)$ in the
sense that $\cH^1(\Omega)\ni v\mapsto \langle q,v\rangle_{\Lp{2}(\Omega)}$ for $q\in\Lp{2}(\Omega)$).
Show that $\divcon$ is continuous and satisfies~\labelcref{eq:definDistDivcon2}.
Furthermore, show \Cref{rem:StrVarFormVsWeakForm}.
\end{exercise}
\begin{exercise}\label{exer:2.25}
Let $\Omega\subseteq \R^d$ be open and bounded, $0<\alpha\leq\beta$, and $(a_n)_{n\in\N}$ in $M(\alpha,\beta;\Omega)$.
Show that for each $f\in\sobH^{-1}(\Omega)$ both the sequence $(u_n)_{n\in\N}$ in $\cH^1(\Omega)$ of unique solutions to
\[
-\divcon a_n \cgrad u_n =f\text{,}
\]
and $(a_n\cgrad u_n)_{n\in\N}$ in $\Lp{2}(\Omega)^d$ are bounded.
 
\end{exercise}
\begin{exercise}[Slight generalisation of the abstract $\div$-$\curl$ lemma]\label{exer:2.2} Let $\mathcal{H}$ be a Hilbert space, $(q_n)_{n\in\N}$, $(r_n)_{n\in\N}$ weakly convergent in $\mathcal{H}$, $\mathcal{H}_0\subseteq \mathcal{H}$ a closed subspace, $\mathcal{H}_1\coloneqq \mathcal{H}_0^\perp$, and $\iota_{0/1}\colon \mathcal{H}_{0/1}\hookrightarrow \mathcal{H}$. If both $\iota_0^* q_n\to \iota_0^*\wlim_{n\to\infty} q_n$ and $\iota_1^* r_n\to \iota_1^*\wlim_{n\to\infty} r_n$ strongly in $\mathcal{H}_0$ and $\mathcal{H}_1$, respectively, then
\[
   \lim_{n\to\infty}\langle q_n, r_n\rangle_{\mathcal{H}} = \langle \wlim_{n\to\infty} q_n, \wlim_{n\to\infty} r_n\rangle_{\mathcal{H}}\text{.}
\]
\end{exercise}

\chapter{Lecture 3}

\section{Introduction}
In the previous lecture, we have shown that $\Htopo$-convergence can be reformulated in terms of an operator topology -- without the need to restrict the coefficients to multiplication operators or the like. This perspective will enable us to study convergence of abstract operator equations in Hilbert spaces. These will form the foundation of the homogenisation results for (time-dependent) partial differential equations at the end of this lecture.

\section{An abstract convergence result}

In order to state and prove a fundamental observation concerning certain operator equations in Hilbert spaces, we quickly introduce some somewhat known notions in a slightly different light, compared to the previous lectures. 

Let $\mathcal{H}_0,\mathcal{H}_1$ be Hilbert spaces and $A\colon \dom(A)\subseteq \mathcal{H}_0 \to \mathcal{H}_1$ densely defined. Then, we define its \emph{adjoint} $A^*\colon \dom(A^*)\subseteq \mathcal{H}_1 \to \mathcal{H}_0$ as follows:
\[
   \dom(A^*)\coloneqq \{ u \in \mathcal{H}_1\mid \exists v\coloneqq A^*u \in \mathcal{H}_0\; \forall x\in \dom(A)\colon \langle Ax, u\rangle_{\mathcal{H}_1} = \langle x, v\rangle_{\mathcal{H}_0} \}.
\]
As $A$ is densely defined, $A^*$ is well-defined. It is also possible to show that $A^*$ is closed. $A$ is called \emph{skew-selfadjoint}, if $A=-A^*$. Note that this requires $\dom (A)=\dom (A^*)$ in particular. We call $A$ \emph{closable} if there exists a closed operator that extends $A$. The minimal such extension, i.e., the intersection of all of them, is denoted by $\overline{A}$. We provide some well-known, elementary formulas for computing the adjoint of (unbounded) operators.
\begin{proposition}\label{prop:compadj} Let $\mathcal{H}_0,\mathcal{H}_1$ be Hilbert spaces, $A\colon \dom(A)\subseteq \mathcal{H}_0 \to \mathcal{H}_1$ densely defined, and $T\in \Lb(\mathcal{H}_0,\mathcal{H}_1)$. Then,
\[
 A^{**}\coloneqq(A^*)^*=\overline{A},\ A^*=\overline{A}^*,\  (A+T)^* = A^*+T^*,\text{ and }(TA)^*=A^*T^*.
\]If, in addition, $T$ is an isomorphism, then $(AT)^*=T^*A^*$.
\end{proposition}
Similar to the case of bounded linear operators, range of $A$ and kernel of $A^*$ are mutually orthogonal, i.e., for a densely defined and closed $A\colon \dom(A)\subseteq \mathcal{H}_0 \to \mathcal{H}_1$,
\[
   \overline{\ran}(A) = \ker(A^*)\text{ and } \overline{\ran}(A^*) = \ker(A)^\bot.
\]
We quickly introduce the following set of abstract operator coefficients: For $0<\alpha\leq\beta$, and $\mathcal{H}$ Hilbert space, we define
\[
   \mathcal{F}(\alpha,\beta; \mathcal{H})\coloneqq \{ T\in \Lb(\mathcal{H}): \Re T\geq \alpha, \Re T^{-1}\geq 1/\beta\}\text{.}
\]
\begin{proposition}\label{prop:wpTA} Let  $\mathcal{H}$ be a Hilbert space, let $A\colon \dom(A)\subseteq \mathcal{H} \to \mathcal{H}$ be skew-selfadjoint, and $T\in \Lb(\mathcal{H})$. If $\Re T\geq c>0$, then $T+A$ is continuously invertible, $\|(T+A)^{-1}\|\leq 1/c$ and $\|A(T+A)^{-1}\|\leq (c+\|T\|)/c$
\end{proposition}
\begin{proof}
We compute, for all $x\in \dom(A)=\dom(T+A)$,
\[
\|x\|_{\mathcal{H}}\|(T+A)x\|_{\mathcal{H}}\geq  \Re \langle x,(T+A)x\rangle_{\mathcal{H}} = \Re \langle x, Tx\rangle_{\mathcal{H}} \geq c\|x\|_{\mathcal{H}}^2.
\]
Hence,
\[
  \forall x\in \dom(A):\|(T+A)x\|_{\mathcal{H}}\geq c\|x\|_{\mathcal{H}}\text{.}
\]
Therefore, $T+A$ is one-to-one and,
by \Cref{exer:1.2}, $\ran(T+A)\subseteq \mathcal{H}$ is closed. Since $\Re T = \Re T^*$, we deduce that $(T+A)^*=T^*-A$ is one-to-one as well. Thus, $\ran(T+A)=\mathcal{H}$, and $T+A$ is continuously invertible with $\|(T+A)^{-1}\|\leq 1/c$. As a consequence, $\|A(T+A)^{-1}\|\leq 1+\|T(T+A)^{-1}\|\leq 1+\|T\|(1/c)$.
\end{proof}

\begin{theorem}\label{thm:abstractSchur} Let  $\mathcal{H}$ be a separable Hilbert space, $0<\alpha\leq\beta$, and let $A\colon \dom(A)\subseteq \mathcal{H} \to \mathcal{H}$ be skew-selfadjoint. Assume that $\dom(A)\cap \ker(A)^\bot \hookrightarrow \mathcal{H}$ is compact. Let $(T_n)_{n\in\N}$ be a sequence in  $\mathcal{F}(\alpha,\beta; \mathcal{H})$. Then, the following conditions are equivalent:
\begin{enumerate}
\item\label{thm:abstractSchur1} There exists $T\in \mathcal{M}(\mathcal{H}_0,\mathcal{H}_1)$ such that $T_n \to T$ in $\tau(\ker(A),\ran(A))$ as $n\to\infty$.
\item\label{thm:abstractSchur2}  There exists $S\in \Lb(\mathcal{H})$ such that $(T_n+A)^{-1}\to S$ in the weak operator topology as $n\to\infty$.
\end{enumerate}
In either case, $T\in \mathcal{F}(\alpha,\beta;\mathcal{H})$ and $S=(T+A)^{-1}$.
\end{theorem}

\section[Schur compact sets]{{{Schur compact sets}}}
Before we turn to the proof of the main abstract result of this lecture, \Cref{thm:abstractSchur}, we provide some additional information about the Schur topology. For this, we quickly recall a fundamental observation for the weak operator topology.

\begin{theorem}\label{thm:factswot} Let $\mathcal{H}_0,\mathcal{H}_1$ be Hilbert spaces. Then, the following statements are true:
\begin{enumerate}[label=(\alph*)]
  \item\label{thm:factswot1} $B_1 \coloneqq \{T\in \Lb(\mathcal{H}_0,\mathcal{H}_1): \|T\|\leq 1\}$ is compact with respect to the weak operator topology.
  \item\label{thm:factswot2} If both $\mathcal{H}_0,\mathcal{H}_1$ are separable, then $B_1$ is metrisable and, thus, sequentially compact.
  \item\label{thm:factswot3} If $\mathcal{H}_0=\mathcal{H}_1$ is separable and $0<\alpha\leq\beta$, then $\mathcal{F}(\alpha,\beta;\mathcal{H}_0)$ is closed under the weak operator topology.
\end{enumerate}
\end{theorem}
\begin{proof}
 \labelcref{thm:factswot1}: First of all, note that by the Riesz--Frech\'et representation theorem, sesquilinear forms and linear operators can be canonically identified. Moreover, it is not difficult to see that the sesquilinear forms on ${\mathcal{H}_0\times \mathcal{H}_1}$ (seen as a subset of $\K^{\mathcal{H}_0\times \mathcal{H}_1}$) are closed with respect to the Euclidean product topology on $\K^{\mathcal{H}_0\times \mathcal{H}_1}$. Thus,
 \begin{equation}\label{eq:TikhonovPointwiseConstrProdTop}
    B_1 \coloneqq \prod_{\phi\in \mathcal{H}_0,\psi\in \mathcal{H}_1} \|\phi\|_{\mathcal{H}_0}\|\psi\|_{\mathcal{H}_1}B_{\K}[0,1] \cap \{ s\colon \mathcal{H}_0\times \mathcal{H}_1\to \K\mid s \text{ sesquilinear}\},
 \end{equation}
 where $B_{\K}[0,1]\coloneqq\{z\in \K: |z|\leq 1\}$, is compact by Tikhonov's theorem.
 
 \labelcref{thm:factswot2}: Let $\mathcal{D}_j\subseteq \mathcal{H}_j$ be countable and dense, and consider the metrisable (\Cref{exer:3.0}) topology $\tau$ induced by only considering $\phi\in \mathcal{D}_0$ and $\psi\in \mathcal{D}_1$ in~\labelcref{eq:TikhonovPointwiseConstrProdTop}. Then, denoting by $\tau_{\mathrm{w}}$ the original weak operator topology from~\labelcref{eq:TikhonovPointwiseConstrProdTop},
 we obtain that
 \[
    (B_1,\tau_{\textnormal{w}})\hookrightarrow (B_1,\tau)
 \]
 is one-to-one, onto, and continuous. As the former space is compact by  \labelcref{thm:factswot1} and the latter Hausdorff due to the density of $\mathcal{D}_j$ in $\mathcal{H}_j$, the above embedding is an actual homeomorphism.
 
  \labelcref{thm:factswot3}: By  \labelcref{thm:factswot2}, $\mathcal{F}(\alpha,\beta;\mathcal{H}_0)\subseteq \beta B_1$ is metrisable if endowed with the weak operator topology. Moreover, for all $\phi\in \mathcal{H}_0$, the mapping
 \[
    \phi_a\colon
    \begin{cases}
    \hfill \Lb^{\mathrm{w}}(\mathcal{H}_0)&\to\quad\R\\
 \hfill T &\mapsto\quad \Re \langle \phi,T\phi\rangle_{\mathcal{H}_0}
 \end{cases}
 \]
 is continuous. This implies that
 \[
    \mathcal{F}_a \coloneqq \bigcap_{\phi\in \mathcal{H}_0} \phi_a^{-1}[[\alpha,\infty)] = \{ T\in \Lb(\mathcal{H}_0): \Re T\geq \alpha\}
 \]
 is closed with respect to the weak operator topology. Next, take $(T_n)_{n\in\N}$ in $\mathcal{F}(\alpha,\beta;\mathcal{H}_0)$ converging to some $T\in \beta B_1$ in the weak operator topology. Then, $T\in \mathcal{F}_a$. Furthermore, the Cauchy--Schwarz inequality yields
 \begin{equation}\label{eq:CauchySchwarzReplacesLimInf}
\frac{1}{\beta}\frac{\abs{\langle T\phi,T_n\phi\rangle_{\mathcal{H}_0}}^2}{\|T\phi\|_{\mathcal{H}_0}^2}\leq   \frac{1}{\beta}\|T_n\phi\|_{\mathcal{H}_0}^2 \leq \Re \langle \phi,T_n\phi\rangle_{\mathcal{H}_0}
 \end{equation}
 for all $\phi\in \mathcal{H}_0$ and $n\in\N$.
 By taking the limits in the inner products only, we prove
 \[
 \frac{1}{\beta}\|T\phi\|_{\mathcal{H}_0}^2\leq \Re\langle \phi,T\phi\rangle_{\mathcal{H}_0}
 \]
 for all  $\phi\in \mathcal{H}_0$, i.e., $\Re T^{-1}\geq 1/\beta$.
\end{proof}

\begin{theorem}\label{thm:compSchur} Let $\mathcal{H}$ be a separable Hilbert space, $\mathcal{H}_0\subseteq \mathcal{H}$ a closed subspace, $\mathcal{H}_1\coloneqq \mathcal{H}_0^\bot$. For $\alpha = (\alpha_{jk})_{j,k\in \{0,1\}}$ in $(0,\infty)^{2\times 2}$ consider
\begin{align*}
   \mathcal{F}(\alpha; \mathcal{H}_0,\mathcal{H}_1) \coloneqq \{ &T\in \mathcal{M}(\mathcal{H}_0,\mathcal{H}_1):\\ 
   & T_{00} \in \mathcal{F}(\alpha_{00}, \alpha_{11}; \mathcal{H}_0), \\
   & T_{10}T_{00}^{-1}  \in \alpha_{10}B_1, T_{00}^{-1}T_{01}  \in \alpha_{01}B_1,  \\
   & T_{11}-T_{10}T_{00}^{-1}T_{01} \in \mathcal{F}(\alpha_{00},\alpha_{11}; \mathcal{H}_1)\}.
\end{align*}
Then, $\mathcal{F}(\alpha; \mathcal{H}_0)$ is (sequentially) compact under $\tau(\mathcal{H}_0,\mathcal{H}_1)$. 
\end{theorem}
\begin{proof}
With the help of \Cref{thm:factswot}, successive selections of subsequences do the job. 
\end{proof}

\section[Proof of Theorem 3.2.3 (I)$\Rightarrow$(II)]{{{Proof of \Cref{thm:abstractSchur} \labelcref{thm:abstractSchur1}$\Rightarrow$\labelcref{thm:abstractSchur2}}}}

\begin{lemma}\label{lem:Assa} Let  $\mathcal{H}$ be a Hilbert space and let $A\colon \dom(A)\subseteq \mathcal{H} \to \mathcal{H}$ be skew-selfadjoint. Let $\iota_0 \colon \ker(A)\hookrightarrow \mathcal{H}$ and $\iota_1 \colon \overline{\ran}(A)\hookrightarrow \mathcal{H}$. Then,
\[
   \begin{pmatrix} \iota_0^* \\ \iota_1^* \end{pmatrix} A    \begin{pmatrix} \iota_0 & \iota_1 \end{pmatrix} = \begin{pmatrix} 0& 0 \\ 0 & \tilde{A} \end{pmatrix},
\]
where 
\[
\tilde{A}\colon\begin{cases}
\hfill \dom(A)\cap \overline{\ran}(A)\subseteq \overline{\ran}(A)&\to\quad\overline{\ran}(A)\\
 \hfill x &\mapsto\quad Ax
 \end{cases}
\]
 is skew-selfadjoint. 
Moreover, the following statements hold:
\begin{enumerate}[label=(\alph*)]
  \item\label{lem:Assa1} If $\ran(A)\subseteq \mathcal{H}$ is closed, then $\tilde{A}$ is continuously invertible.
  \item\label{lem:Assa2} If $\dom(A)\cap \ker(A)^\bot \hookrightarrow \mathcal{H}$ is compact, then $\ran(A)\subseteq \mathcal{H}$ is closed and $\tilde{A}^{-1}$ is compact.
  \end{enumerate}
\end{lemma}
\begin{proof}
First of all, note that $\overline{\ran}(A) = \ker(A)^\bot$ as $A$ is skew-selfadjoint. Clearly, we have $A\iota_0 = 0$. Moreover, if $x\in \overline{\ran}(A)$ then $\iota_1 x \in \dom(A)$ if and only if $x\in \overline{\ran}(A)\cap \dom(A)$. Hence, 
\[
   A    \begin{pmatrix} \iota_0 & \iota_1 \end{pmatrix} =       \begin{pmatrix} 0&  A \iota_1  \end{pmatrix}.
\]
Next, since for all $x\in \overline{\ran}(A)\cap \dom(A)$, $A\iota_1 x \in \ran (A) \subseteq\ker(A)^\bot$, we deduce $\iota_1^*A\iota_1 x = A\iota_1 x$ and $\iota_0^*A\iota_1 x =0$. Thus, the first claim follows with $\tilde{A} = \iota_1^*A\iota_1$.

Note that $\overline{\ran}(A)\cap \dom(A)$ is dense in $\overline{\ran}(A)$. Indeed, as $\dom(A)$ is dense in $\mathcal{H}$, for $y\in \overline{\ran}(A)$, we find a sequence $(x_n)_{n\in\N}$ in $\dom(A)$ such that $x_n\to y$ in $\mathcal{H}$. We define $y_n \coloneqq x_n - \iota_0\iota_0^*x_n \in \dom(A)$ for $n\in\N$. Then, $y_n \in \overline{\ran}(A)\cap \dom(A)$ for $n\in\N$ and $y_n\to y \in \overline{\ran}(A)$ as $n\to\infty$.
Next, it is not difficult to see that $-\iota_1^*A\iota_1 \subseteq (\iota_1^*A\iota_1)^*$. For the other inclusion, let $u \in \dom \big((\iota_1^*A\iota_1)^*\big)\subseteq \overline{\ran}(A)$ and $v = (\iota_1^*A\iota_1)^*u\subseteq \overline{\ran}(A)$. Then, for all $x\in \dom(\iota_1^*A\iota_1)$, we compute
\[
  \langle x, (\iota_1^*A\iota_1)^*u\rangle_{\overline{\ran}(A)} = \langle A\iota_1 x,\iota_1 u\rangle_{\mathcal{H}}\text{.}
\] In particular, if $w\in \ker(A)$, we find
\[
 \langle x, (\iota_1^*A\iota_1)^*u\rangle_{\overline{\ran}(A)} =  \langle \iota_1 x, \iota_1 (\iota_1^*A\iota_1)^*u\rangle_{\mathcal{H}} =  \langle \iota_1 x+\iota_0 w, \iota_1 (\iota_1^*A\iota_1)^*u\rangle_{\mathcal{H}}\text{,}
\]and, as $ A\iota_1 x =  A(\iota_1 x+\iota_0 w)$, we obtain
\[
  \forall z\in \dom(A) : \langle z, \iota_1(\iota_1^*A\iota_1)^*u\rangle_{\mathcal{H}} = \langle A z, \iota_1 u\rangle_{\mathcal{H}}\text{.}
\]Thus, $\iota_1 u \in \dom(A^*)$ and $A^*\iota_1 u =  \iota_1(\iota_1^*A\iota_1)^*u$. Hence, $\iota_1^* A^*\iota_1 u =(\iota_1^*A\iota_1)^*u=v$.

\labelcref{lem:Assa1}: If $\ran(A)=\overline{\ran}(A)$, then $\tilde{A}$ is one-to-one and onto. As $\tilde{A}=-\tilde{A}^*$, $\tilde{A}$ is closed. Thus, $\tilde{A}^{-1}$ is closed and defined on the whole of $\ran(A)$, which is a Hilbert space, so the closed graph theorem yields continuity of $\tilde{A}^{-1}$.

\labelcref{lem:Assa2}: Closedness of $\ran(A)$ follows from \Cref{exer:2.1}. Since $\tilde{A}^{-1}$ is continuous, it also maps continuously into $\dom(A)\cap \ran(A)=\dom(A)\cap \ker(A)^\bot$, endowed with the graph norm. Since $\dom(A)\cap \ran(A)\hookrightarrow \mathcal{H}$ is compact, so is $\tilde{A}^{-1}\colon \mathcal{H}\to \mathcal{H}$. 
\end{proof}

\begin{lemma}\label{lem:convcomp} Let $\mathcal{H}$ be a Hilbert space, $0<\alpha\leq\beta$, $A\colon \dom(A)\subseteq \mathcal{H}\to\mathcal{H}$ skew-selfadjoint, one-to-one with $\dom(A)\hookrightarrow \mathcal{H}$ compact. If $(T_n)_{n\in\N}$ in $\mathcal{F}(\alpha,\beta;\mathcal{H})$ converges in the weak operator topology to some $T\in \mathcal{F}(\alpha,\beta;\mathcal{H})$, then for all $(f_n)_{n\in\N}$ weakly convergent in $\mathcal{H}$, we deduce
\[
   (T_n+A)^{-1} f_n \to (T+A)^{-1} (\wlim_{n\to\infty} f_n)\in \mathcal{H}
\]
as $n\to\infty$.
\end{lemma}
\begin{proof}
Let $(f_n)_{n\in\N}$ be weakly convergent and set $u_n\coloneqq (T_n+A)^{-1}f_n$ for $n\in\N$, and $f\coloneqq \wlim_{n\to\infty} f_n$. Then, by the boundedness of $(f_n)_{n\in\N}$ (uniform boundedness principle) and by \Cref{prop:wpTA}, $(u_n)_{n\in\N}$ belongs to a bounded set in $\dom(A)$. Thus, we choose a subsequence (not relabelled) so that $v\coloneqq \wlim_{n\to\infty} u_n \in \dom(A)$. By compactness of $\dom(A)\hookrightarrow 
\mathcal{H}$, we infer that $v=\lim_{n\to\infty} u_n \in \mathcal{H}$. Hence, for $n
\in \N$,
\[
  f_n =  T_n u_n +A u_n \text{ and } f = Tv+Av,
\]
where we took weak limits. Since $v= (T+A)^{-1}f$, the whole sequence $(u_n)_{n\in\N}$ weakly converges to $v$ in $\dom(A)$ and, thus again by compactness, strongly in $\mathcal{H}$, which yields the claim.
\end{proof}

\begin{lemma}\label{lem:schurpd}
Let $\mathcal{H}$ be a Hilbert space, $\mathcal{H}_0\subseteq \mathcal{H}$ a closed subspace, $\mathcal{H}_1\coloneqq \mathcal{H}_0^\bot$, $0<\alpha\leq\beta$, and $T\in \mathcal{F}(\alpha,\beta;\mathcal{H})$. Then, setting $T_{jk}\coloneqq \iota_j^* T\iota_k$ for $j,k\in \{0,1\}$, we have
\[
    T_S\coloneqq T_{11}-T_{10}T_{00}^{-1}T_{01} \in \mathcal{F}(\alpha,\beta;\mathcal{H}_1).
\]
\end{lemma}
\begin{proof}
We compute
\begin{align*}
  &  \begin{pmatrix} 1 &0 \\  - (T_{00}^{-1}T_{01})^* & 1 \end{pmatrix}  \begin{pmatrix} T_{00} & T_{01} \\ T_{10} & T_{11} \end{pmatrix}
    \begin{pmatrix} 1 & -T_{00}^{-1}T_{01} \\0 & 1 \end{pmatrix}
     \\
     & =
     \begin{pmatrix} 1 &0 \\  - (T_{00}^{-1}T_{01})^* & 1 \end{pmatrix}  \begin{pmatrix} T_{00} & 0 \\ T_{10} & T_S\end{pmatrix}
     \\& =
\begin{pmatrix} T_{00} & 0 \\ - (T_{00}^{-1}T_{01})^*T_{00}+T_{10} & T_S \end{pmatrix} \eqqcolon R\text{.}
\end{align*}
Hence, for $\psi\in \mathcal{H}_1$ and $Q\coloneqq \begin{pmatrix} 1 & -T_{00}^{-1}T_{01} \\0 & 1 \end{pmatrix}$, we compute
\begin{align*}
  \Re \langle \psi,T_S \psi \rangle_{\mathcal{H}_1} & =
  \Re \langle \begin{pmatrix} 0 \\ \psi\end{pmatrix},R \begin{pmatrix} 0 \\ \psi\end{pmatrix}\rangle_{\mathcal{H}} \\
  & =
  \Re \langle \begin{pmatrix} 0 \\ \psi\end{pmatrix},  Q^*TQ\begin{pmatrix} 0 \\ \psi\end{pmatrix}\rangle_{\mathcal{H}} \\
  & =
  \Re \langle Q\begin{pmatrix} 0 \\ \psi\end{pmatrix},   T
   Q\begin{pmatrix} 0 \\ \psi\end{pmatrix}\rangle_{\mathcal{H}} \geq \alpha \|Q\begin{pmatrix} 0 \\ \psi\end{pmatrix}\|_{\mathcal{H}}^2 \geq \alpha \|\psi\|_{\mathcal{H}_1}^2.
\end{align*}
Note that, quite elementarily, it follows that $\Re T_{00}\geq \alpha$. Next, we observe that $T\in \mathcal{F}(\alpha,\beta;\mathcal{H})$ is equivalent to $T^{-1} \in \mathcal{F}(\frac{1}{\beta}, \frac{1}{\alpha};\mathcal{H})$. By interchanging roles of $\mathcal{H}_0$ and $\mathcal{H}_1$, we infer that $\Re (T^{-1})_{11}\geq 1/\beta$. Since $(T^{-1})_{11} = T_S^{-1}$, the claim follows.
\end{proof}

\begin{proof}[Proof of \Cref{thm:abstractSchur} \labelcref{thm:abstractSchur1}$\Rightarrow$\labelcref{thm:abstractSchur2}] By \Cref{lem:Assa}, $\ran(A)\subseteq \mathcal{H}$ is closed. We set $\iota_0 \colon \ker(A)\hookrightarrow \mathcal{H}$ and $\iota_1\colon \ran(A)\hookrightarrow \mathcal{H}$. Let $f\in \mathcal{H}$ and define $u_n\coloneqq (T_n+A)^{-1} f$ for all $n\in \N$.
Then, with $T_{(n),jk} \coloneqq \iota_j^*T_{(n)}\iota_k$ for $j,k\in \{0,1\}$ and $n\in\N$, we compute for all $n\in \N$:
\[
 \begin{pmatrix} f_0 \\ f_1 \end{pmatrix} \coloneqq \begin{pmatrix} \iota_0^* \\ \iota_1^*\end{pmatrix} f =  \begin{pmatrix} \iota_0^* \\ \iota_1^*\end{pmatrix}(T_n+A)\begin{pmatrix} \iota_0 & \iota_1\end{pmatrix}\begin{pmatrix} \iota_0^* \\ \iota_1^*\end{pmatrix}u_n,
\]
that is, by virtue of  \Cref{lem:Assa},
\begin{equation}\label{eq:01syst}
\begin{pmatrix} f_0 \\ f_1 \end{pmatrix} = \big(\begin{pmatrix} T_{n,00} & T_{n,01} \\ T_{n,10} & T_{n,11}\end{pmatrix} + \begin{pmatrix} 0 & 0 \\ 0 & \tilde{A} \end{pmatrix}\big) \begin{pmatrix} u_{n,0} \\ u_{n,1} \end{pmatrix}
\end{equation}with $u_{n,j} \coloneqq \iota_j^*u_n$ for $n\in\N$ and $j\in\{0,1\}$. Once again by virtue of  \Cref{lem:Assa}, $\tilde{A}$ is skew-selfadjoint and has compact inverse. 
We multiply \labelcref{eq:01syst} from the left by $\begin{pmatrix} 1 & 0 \\ -T_{n,10}T_{n,00}^{-1} & 1\end{pmatrix}$ to obtain ($T_{n,S}\coloneqq T_{n,11}-T_{n,10}T_{n,00}^{-1}T_{n,01}$ for $n\in\N$)
\[
 \begin{pmatrix} f_0 \\ f_1 -T_{n,10}T_{n,00}^{-1}f_0  \end{pmatrix} = \big(\begin{pmatrix} T_{n,00} & T_{n,01} \\ 0 & T_{n,S}\end{pmatrix} + \begin{pmatrix} 0 & 0 \\ 0 & \tilde{A} \end{pmatrix}\big) \begin{pmatrix} u_{n,0} \\ u_{n,1} \end{pmatrix}\text{.}\]
 Multiplying the first line by $T_{n,00}^{-1}$, we deduce
 \[
  \begin{pmatrix} u_{n,0} \\ u_{n,1} \end{pmatrix} = \begin{pmatrix} T_{n,00}^{-1} f_0 - T_{n,00}^{-1}T_{n,01}u_{n,1}\\ \big(T_{n,S}+\tilde{A}\big)^{-1}\big(f_1 -T_{n,10}T_{n,00}^{-1}f_0 \big)\end{pmatrix},
 \]where we used that $\tilde{A}$ is skew-selfadjoint (\Cref{lem:Assa}), $T_{n,S}\in\mathcal{F}(\alpha,\beta;\mathcal{H}_1)$ (\Cref{lem:schurpd}), and \Cref{prop:wpTA} to deduce that $T_{n,S}+\tilde{A}$ is continuously invertible for $n\in\N$.  Using \Cref{thm:factswot}~\labelcref{thm:factswot3}, we define 
 \[
    u_1 \coloneqq  \big(T_S+\tilde{A}\big)^{-1}\big(f_1 -T_{10}T_{00}^{-1}f_0 \big)
 \]
 and
 \[
    u_0 \coloneqq T_{00}^{-1} f_0 - T_{00}^{-1}T_{01}u_{1}.
 \]
 Then, by Schur convergence and \Cref{lem:convcomp}, $u_{n,1}\to u_1 \in \ran(A)$. Thus, $u_{n,0}\rightharpoonup u_0\in \ker(A)$, where we used \Cref{rem:fafact} with $\mathcal{Y}=\mathbb{K}$. A suitable rearrangement yields
 \[
     (T+A) u = f,\text{ where }u=\iota_0u_0+\iota_1u_1,
 \]and thus $(T_n+A)^{-1}\to (T+A)^{-1}$ in the weak operator topology.
\end{proof}

\section[Proof of Theorem 3.2.3 (II)$\Rightarrow$(I)]{{{Proof of \Cref{thm:abstractSchur} \labelcref{thm:abstractSchur2}$\Rightarrow$\labelcref{thm:abstractSchur1}}}}

The next proposition is a surprising result concerning the convergence of operators of the form $(T_n+A)^{-1}$ with $A$ skew-selfadjoint without the compactness assumption. However, in applications without the compactness assumption, it cannot be guaranteed that $(T_n)_{n\in\N}$ itself converges in any sense.

\begin{proposition}\label{prop:BuVr} Let $\mathcal{H}$ be a Hilbert space, $A\colon \dom(A)\subseteq \mathcal{H}\to \mathcal{H}$ skew-selfadjoint, and $(T_n)_{n\in\N}$ a sequence in $\mathcal{F}(\alpha,\beta;\mathcal{H})$. If $(T_n+A)^{-1}$ converges to some $S\in \Lb(\mathcal{H})$ in the weak operator topology, then there exists $T\in \mathcal{F}(\alpha,\beta;\mathcal{H})$ such that $S=(T+A)^{-1}$. This $T$ is unique in the following sense: If $\tilde{T} \in \Lb(\mathcal{H})$ satisfies $S^{-1}=\tilde{T}+A$, then $T=\tilde{T}$.
\end{proposition}
\begin{proof}
 For $f\in \mathcal{H}$ define $K\in \Lb(\mathcal{H})$ via $Kf\coloneqq f-ASf$. This is well-defined as $(T_n+A)^{-1}$ defines a uniformly bounded sequence in $\Lb(\mathcal{H},\dom(A))$ by \Cref{prop:wpTA}, i.e., $S\in \Lb(\mathcal{H},\dom(A))$. Next, fix $f\in \mathcal{H}$ and let $u_n\coloneqq (T_n+A)^{-1}f$ for $n\in\N$. For the same reason that we argued when defining $K$, $Au_n\rightharpoonup AS f \in \mathcal{H}$ as $n\to\infty$. Since
 \[
     T_n u_n +Au_n = f
 \]
for $n\in \N$,  we get $T_nu_n\rightharpoonup f- ASf = Kf \in \mathcal{H}$ as $n\to \infty$. In particular,
\begin{multline*}
   \Re \langle u_n ,T_n u_n \rangle_{\mathcal{H}} =    \Re \langle  u_n ,(T_n +A) u_n \rangle_{\mathcal{H}} = \Re \langle u_n,f\rangle_{\mathcal{H}}\\ \to \Re \langle Sf,f\rangle_{\mathcal{H}} = \Re \langle Sf,f-ASf\rangle_{\mathcal{H}} =\Re \langle Sf,Kf\rangle_{\mathcal{H}}
\end{multline*}
as $n\to\infty$.
Since $u_n\rightharpoonup Sf$, we can argue similarly to~\labelcref{eq:CauchySchwarzReplacesLimInf} and obtain
\[
    \alpha \|Sf\|_{\mathcal{H}}^2 = \lim_{n\to\infty}\alpha \frac{\abs{\langle Sf, u_n\rangle_{\mathcal{H}}}^2}{\|Sf\|_{\mathcal{H}}^2}\leq \lim_{n\to\infty} \Re \langle u_n,T_n u_n\rangle_{\mathcal{H}} = \Re \langle Sf,Kf\rangle_{\mathcal{H}}\text{,}
\]
as well as
\[
    \frac{1}{\beta} \|Kf\|_{\mathcal{H}}^2 = \lim_{n\to\infty}\frac{1}{\beta} \frac{\abs{\langle Kf, T_n u_n\rangle_{\mathcal{H}}}^2}{\|Kf\|_{\mathcal{H}}^2}\leq \lim_{n\to\infty} \Re \langle u_n,T_n u_n\rangle_{\mathcal{H}} = \Re \langle Sf,Kf\rangle_{\mathcal{H}}\text{.}
\]
Hence, $f\in \ker(S)$ implies $f\in\ker (K)$, and further $0=Kf = f-ASf = f$, i.e., $S$ is one-to-one. Next, take $g\in \ran(S)^\perp$. We have
\[
 0 = \Re  \langle g, Sg\rangle_{\mathcal{H}} =\Re \langle g-ASg, Sg\rangle_{\mathcal{H}} = \Re \langle Kg, Sg\rangle_{\mathcal{H}} \geq \alpha \|Sg\|_{\mathcal{H}}^2.
\]
As $S$ is one-to-one, $Sg=0$ implies $g=0$, i.e., $\ran(S)\subseteq\mathcal{H}$ is dense. 

We define $T\colon \ran(S)\subseteq \mathcal{H}\to \mathcal{H}$ via $T(Sf)\coloneqq Kf$. $T$ is well-defined, densely defined, and continuous. Therfore, $T$  uniquely extends to a bounded linear operator in $\Lb(\mathcal{H})$. We have already proven the inequality
\[
  \forall f\in \mathcal{H}: \alpha \|Sf\|_{\mathcal{H}}^2\leq \Re \langle Sf,Kf\rangle_{\mathcal{H}} = \Re \langle Sf,T(Sf)\rangle_{\mathcal{H}}\text{.}
\]This implies $\Re T\geq \alpha$ and, thus, the existence of a bounded inverse. Furthermore, we have already shown 
\[
   \forall f\in \mathcal{H}: \frac{1}{\beta} \|T(Sf)\|_{\mathcal{H}}^2 =   \frac{1}{\beta} \|Kf\|_{\mathcal{H}}^2\leq  \Re \langle Kf,Sf\rangle_{\mathcal{H}} = \Re \langle T(Sf),Sf\rangle_{\mathcal{H}},
\]
which eventually yields $\Re T^{-1}\geq 1/\beta$. Hence, $T\in \mathcal{F}(\alpha,\beta;\mathcal{H})$.

For the uniqueness statement, we see that $\tilde{T}+A={T}+A$ implies $\tilde{T}\restriction_{\dom(A)}={T}\restriction_{\dom(A)}$. By continuity of both $\tilde{T}$ and $T$, that is enough to show their identity on the whole space $\mathcal{H}$ because $\dom(A)$ is dense.
\end{proof}

\begin{proof}[Proof of \Cref{thm:abstractSchur} \labelcref{thm:abstractSchur2}$\Rightarrow$\labelcref{thm:abstractSchur1}]
By \Cref{prop:BuVr}, we find $T\in \mathcal{F}(\alpha,\beta;\mathcal{H})$ such that $S=(T+A)^{-1}$. By \Cref{thm:compSchur} and \Cref{exer:3.05}, we may choose a subsequence of $(T_n)_{n\in\N}$ (not relabelled) such that $T_n$ converges to some $\tilde{T}\in\mathcal{M}(\mathcal{H}_0,\mathcal{H}_1)$ in $\tau(\ker(A),\ran(A))$ as $n\to\infty$. By \labelcref{thm:abstractSchur1}$\Rightarrow$\labelcref{thm:abstractSchur2}, we deduce $S=(\tilde{T}+A)^{-1}$. Hence, $\tilde{T}=T$.
\end{proof}

\section{Applications}\label{sec:ex}


The remainder of these lecture notes is concerned with applications to partial differential equations. The partial differential equations we will be looking at will be set up in time-space for a spatial Hilbert space $\mathcal{H}$, i.e., an $\mathcal{H}$-valued $\Lp{2}$-space. The partial differential equations will be of the form
\[
   (\partial_t M_0 + M_1 +A)U=F
\]
for a given right-hand side $F$, an unknown $U$, both functions on the time-space, as well as given bounded linear operators $M_0,M_1$ and a skew-selfadjoint operator $A$, all three acting with respect to the spatial component $\mathcal{H}$. The operator $\partial_t$ stands for the (weak) time-derivative. Concentrating on homogenisation results in these lectures, we take the fact for granted that a convenient Hilbert space setting, where the above abstract equation can be established and studied, exists. Note that in this setting initial conditions appear in a certain hidden way at time $t=-\infty$. In order to apply the results from the present lecture notes, we need the following theorem.

\begin{theorem}\label{thm:FourLapWOT} Under suitable positive definitenss conditions on $\lambda \mapsto \lambda M_{n,0} + M_{n,1}$ for $n\in\N$, sequences $(M_{n,0})_{n\in\N}$, $(M_{n,1})_{n\in\N}$ in $\Lb(\mathcal{H})$, and an operator-valued, holomorphic function $\lambda\mapsto M(\lambda)$ from a subset of $\C$ to $\Lb(\mathcal{H})$, the following conditions are equivalent:
\begin{enumerate}
  \item $\overline{(\partial_t M_{n,0} + M_{n,1} +A)}^{-1}\to \overline{(\partial_t M(\partial_t) +A)}^{-1}$ in the weak operator topology,
  \item for all sufficiently large real numbers $\lambda$, we have
  \[
  \overline{(\lambda M_{n,0} + M_{n,1} +A)}^{-1}\to \overline{(\lambda M(\lambda) +A)}^{-1}
  \] in the weak operator topology.
\end{enumerate}
\end{theorem}

In the latter theorem `suitable' always implies that for any fixed sufficiently large $\lambda>0$, we have $\Re (\lambda M_{n,0} + M_{n,1})\geq c$ for some $c>0$. It is this property we are going to use in the following. Moreover, $M(\partial_t)$ is interpreted as operator-valued functional calculus for the time-derivative operator. In order to understand the following parts it is only necessary to identify $M(\lambda)$ for any fixed $\lambda$ as a bounded linear operator in $\mathcal{H}$.

\subsection{Thermoelasticity}

For the sake of brevity and clarity, we focus on scalar elasticity only. That allows, for instance, to address the thermo-elastic properties of a membrane. We deem only the motions of the particles perpendicular to the null-position of the membrane relevant.  We write (a simplified version of) the equations in question as follows:
\[
   \left( \partial_t M_0 +M_1  + \begin{pmatrix} 0 & \div & 0 & 0 \\ \grad_0 & 0 & 0 & 0 \\ 0 & 0 & 0 & \div \\ 0 & 0 & \grad_0 & 0 \end{pmatrix} \right)\begin{pmatrix} v \\ T \\ \theta\\ Q \end{pmatrix} = \begin{pmatrix} f \\ 0 \\ g\\  0 \end{pmatrix},
\] 
where we assume the clamped boundary conditions of fixing the velocity at the boundary in $0$ position for the velocity $v$ as well as boundary conditions of total cooling to $0$ Kelvin at the boundary for the heat $\theta$. Moreover, $\div\coloneqq-\grad_0^*$ as an unbounded operator in $\Lp{2}(\Omega)$, and $T$ and $Q$ are the stress and the heat flux respectively. The coupling of heat and elastic part is visible in the definitions of $M_0$ and $M_1$ given by
\begin{equation}\label{eq:M0M1}
   M_0 \coloneqq  \begin{pmatrix} \rho_0 & 0 & 0 & 0 \\ 0 & C^{-1} & C^{-1}\Gamma & 0 \\ 0 & \Gamma^*C^{-1} & w+\Gamma^*C^{-1}\Gamma & 0 \\ 0 & 0 & 0 & 0 \end{pmatrix}, \; M_1 \coloneqq \begin{pmatrix} 0 & 0 & 0 & 0 \\ 0 & 0 & 0 & 0 \\ 0 & 0 & 0 & 0 \\ 0 & 0 & 0 & \kappa^{-1}  \end{pmatrix},
\end{equation}
where $C$ and $\kappa$ are material specific matrices and $\rho_0$ and $w$ material specific scalars describing the physical properties of the underlying materials. $\Gamma$ is a constant (scalar) parameter. In order to treat homogenisation problems, we need to consider sequences of $C$, $\kappa$, $w$, and $\rho_0$. However, before doing that, we quickly address skew-selfadjointness of the operator containing the spatial derivatives as part of an exercise.

\begin{proposition}\label{prop:adjblock} Let $\mathcal{H}_0,\mathcal{H}_1$ be Hilbert spaces, and both $B\colon \dom(B)\subseteq \mathcal{H}_0\to \mathcal{H}_1$ and $C\colon \dom(C)\subseteq \mathcal{H}_1\to \mathcal{H}_0$ densely defined and closed. Then
\[
    \begin{pmatrix} 0 & B \\ C & 0
    \end{pmatrix}^* =     \begin{pmatrix} 0 & C^* \\ B^* & 0
    \end{pmatrix}.
\]

\end{proposition}

Finally, before we can apply our theoretical findings to our first example, we have to convince ourselves of the following result, which will be left as an (elementary) exercise.

\begin{lemma}\label{lem:inv} Let $\mathcal{H}$ be a Hilbert space, $\mathcal{H}_0\subseteq \mathcal{H}$ a closed subspace, $\mathcal{H}_1\coloneqq \mathcal{H}_0^\bot$. Then
\[
    \Big(\mathcal{M}(\mathcal{H}_0,\mathcal{H}_1),\tau(\mathcal{H}_0,\mathcal{H}_1)\Big) \ni T\mapsto T^{-1} \in  \Big([\mathcal{M}(\mathcal{H}_0,\mathcal{H}_1)]^{-1},\tau(\mathcal{H}_1,\mathcal{H}_0)\Big)
\]
is a homoemorphism.
\end{lemma}

\begin{theorem} Let $\Omega\subseteq\R^d$ be open and bounded, $0<\alpha\leq\beta$. Assume that $(C_n)_{n\in\N}$ and $(\kappa_n)_{n\in\N}$  $\Htopo$-converge to $C$ and $\kappa$ in $M(\alpha,\beta;\Omega)$, respectively. Moreover, let $(w_n)_{n\in\N}$ and $(\rho_{0,n})_{n\in\N}$ satisfy $\alpha\leq w_n\leq \beta$ and $\alpha\leq \rho_{0,n}\leq \beta$ respectively for all $n\in 
N$, and assume they $\sigma(L_\infty(\Omega),L_1(\Omega))$-converge to some $w$ and $\rho_0$. Then, for all $\lambda\in \R$ sufficiently large
\[
   (\lambda M_{0,n}+M_{1,n}+A)^{-1}\to    (\lambda M_{0}+M_{1}+A)^{-1}
\]
in the weak operator topology, where $A\coloneqq\begin{psmallmatrix} 0 & \div & 0 & 0 \\ \grad_0 & 0 & 0 & 0 \\ 0 & 0 & 0 & \div \\ 0 & 0 & \grad_0 & 0 \end{psmallmatrix}$, and $M_{0,n}$ and $M_{1,n}$ are given as in \labelcref{eq:M0M1} for $n\in\N$.
\end{theorem}
\begin{proof}
We want to apply \Cref{thm:abstractSchur}.
First of all $A$ is skew-selfajoint using \Cref{prop:adjblock}.
Next, consider 
\[
S \coloneqq \begin{pmatrix} 1 & 0 & 0 & 0 \\ 0 & 1 & 0 & 0 \\ 0 & -\Gamma^* & 1 & 0 \\ 0 & 0 & 0 & 1 \end{pmatrix}.
\]
Then
\[
   SM_0 S^* =  \begin{pmatrix} \rho_0 & 0 & 0 & 0 \\ 0 & C^{-1} & 0 & 0 \\ 0 &0  & w  & 0 \\ 0 & 0 & 0 & 0 \end{pmatrix},\ SM_1S=M_1
\]and
\begin{align*}
   SAS^* &=\begin{pmatrix} 1 & 0 & 0 & 0 \\ 0 & 1 & 0 & 0 \\ 0 & -\Gamma^* & 1 & 0 \\ 0 & 0 & 0 & 1 \end{pmatrix}\begin{pmatrix} 0 & \div & 0 & 0 \\ \grad_0 & 0 & 0 & 0 \\ 0 & 0 & 0 & \div \\ 0 & 0 & \grad_0 & 0 \end{pmatrix}\begin{pmatrix} 1 & 0 & 0 & 0 \\ 0 & 1 & -\Gamma& 0 \\ 0 &0 & 1 & 0 \\ 0 & 0 & 0 & 1 \end{pmatrix} \\
 & =\begin{pmatrix} 0 & \div & -\div\Gamma & 0 \\ \grad_0 & 0 & 0 & 0 \\ -\Gamma^*\grad_0 & 0 & 0 & \div \\ 0 & 0 & \grad_0 & 0 \end{pmatrix}.
\end{align*}
Since $S$ is an isomorphism, $SAS^*$ retains its skew-selfadjointness from the one of $A$. 
Now, $\ran(\div) = \Lp{2}(\Omega)$. Hence, 
\[
   \ran(SAS^*) = \Lp{2}(\Omega)\oplus \mathfrak{g}_0(\Omega)\oplus \Lp{2}(\Omega)\oplus \mathfrak{g}_0(\Omega).
\]
Thus, \[\ker(SAS^*)=\ran(SAS^*)^\bot = \{0\}\oplus \mathfrak{g}_0(\Omega)^\bot \oplus \{0\} \oplus \mathfrak{g}_0(\Omega)^\bot.\] 
Hence, by \Cref{thm:abstractSchur}, we need to show that
\[
\lambda SM_{0,n}S^*+M_{1,n} \to \lambda SM_{0}S^*+M_{1}\text{ in }\tau(\ker(A),\ran(A))
\]for all sufficiently large $\lambda\in \R$. 
We need to confirm that
\[
   \rho_{0,n} \stackrel{\tau(\{0\},L_2(\Omega))}\to \rho_0\text{ and }   w_{0,n} \stackrel{\tau(\{0\},L_2(\Omega))}\to w
\]as well as
\[
   C^{-1}_n \stackrel{\tau(\mathfrak{g}_0(\Omega)^\bot,\mathfrak{g}_0(\Omega))}\to C^{-1}\text{ and }   \kappa_n^{-1} \stackrel{\tau(\mathfrak{g}_0(\Omega)^\bot,\mathfrak{g}_0(\Omega))}\to \kappa^{-1}.
\]
The arguments in the former two cases and in the latter two cases are each time similar. Thus, we pvovide the details only for $\rho_{0,n}$ and $C_n^{-1}$, respectively. By \Cref{ex:trivialcases}, the above holds, if $\rho_{0,n}\to \rho_0$ in the weak operator topology understanding the corresponding functions as multiplication operators. Thus, by \Cref{exer:1.3}, the assumption on $\sigma(L_\infty(\Omega),L_1(\Omega))$-convergence is, however, sufficient.

For $(C_n^{-1})_n$, we apply \Cref{lem:inv}, to deduce that the needed convergence is equivalent to $(C_n)_n$ to be $\tau(\mathfrak{g}_0(\Omega),\mathfrak{g}_0(\Omega)^\bot)$-convergent to $C$. This, however, follows from \Cref{thm:HnlH}.
\end{proof}

\subsection{Maxwell's equations}

In this section, we consider a homogenisation problem for Maxwell's equations. More precisely, we consider
\[
  \left(\partial_t \begin{pmatrix} \varepsilon & 0 \\ 0 & \mu \end{pmatrix} + \begin{pmatrix} \sigma & 0 \\ 0 & 0 \end{pmatrix} + \begin{pmatrix} 0 & - \curl \\ \curl_0 & 0 \end{pmatrix} \right)\begin{pmatrix} E \\ H\end{pmatrix} = \begin{pmatrix} - J\\ 0 \end{pmatrix},
\]
where $\varepsilon,\mu,\sigma$ are suitable coefficients. We define $\curl$ and $\curl_0$ weakly in a similar way as $\grad$ and $\cgrad$. Note that
both  $\curl$ and $\curl_0$ are extensions of the differential operator $\nabla \times$ on $\Cc(\Omega)$ and that $\curl= \curl_0^*$.
We recall $\div = -\grad_0^*$ and we define $\div_0\coloneqq -\grad^*$.
Again, in order to discuss things properly a standard result on compactness is needed. For this, we say that 
a bounded $\Omega$ is a \emph{Lipschitz domain}, that is, $\partial\Omega$ can be locally written as a graph of a Lipschitz function.

\begin{theorem}[Picard--Weber--Weck selection theorem] Let $\Omega\subseteq\R^3$ be open and bounded. If $\Omega$ is a Lipschitz domain, then both embeddings (with respect to the sum of the graph norms)
\[
   \dom(\curl_0)\cap \dom(\div)\hookrightarrow \Lp{2}(\Omega)^3\text{ and }   \dom(\curl)\cap \dom(\div_0)\hookrightarrow \Lp{2}(\Omega)^3
\]
are compact.
\end{theorem}
\begin{remark} $\dom(\curl)\cap \dom(\div_0) \hookrightarrow \Lp{2}(\Omega)^3$ compact implies that also $\dom(\curl)\cap \ker(\curl)^\bot \hookrightarrow \Lp{2}(\Omega)^3$ is compact. Indeed, this follows from $\ker(\curl)^\bot=\overline{\ran}(\curl_0)\subseteq \ker(\div_0)\subseteq \dom(\div_0)$ as one quickly verifies by Schwarz' lemma and elementary formulas for smooth, compactly supported vector fields (also, cf.~\Cref{thm:Helmholtz} below). Thus, for Lipschitz domains $\Omega$, both $\ran(\curl)$ and $\ran(\curl_0)$ are closed by \Cref{exer:2.1}.
\end{remark}
The main result is the following:
\begin{theorem}\label{thm:homomaxwell} Let $0<\alpha\leq\beta$, $\Omega\subseteq \R^3$ be an open, bounded Lipschitz domain. Let $\lambda>0$ and $(\varepsilon_n)_{n\in\N}, (\mu_n)_{n\in\N}$, $(\sigma_n)_{n\in\N}$ in $\Lp{\infty}(\Omega)^{3\times 3}$ be uniformly bounded. If $\lambda \varepsilon_n+ \sigma_n,\mu_n \in M(\alpha,\beta;\Omega)$ for $n\in\N$ and, as $n\to\infty$,
\[
    \lambda \varepsilon_n+ \sigma_n \stackrel{\Htopo}{\to}\varepsilon(\lambda),\text{ and } \mu_n \stackrel{\Htopo}{\to} \mu.
\]
Then,
\[
\left(  \lambda \begin{pmatrix} \varepsilon_n & 0 \\ 0 & \mu_n \end{pmatrix} + \begin{pmatrix} \sigma_n & 0 \\ 0 & 0 \end{pmatrix} + \begin{pmatrix} 0 & - \curl \\ \curl_0 & 0 \end{pmatrix} \right)^{-1}\to
\left(   \begin{pmatrix} \varepsilon(\lambda) & 0 \\ 0 & \mu \end{pmatrix}+\begin{pmatrix} 0 & - \curl \\ \curl_0 & 0 \end{pmatrix} \right)^{-1}
\]in the weak operator topology.
\end{theorem}

In order to apply our main abstract theorem, we require some results from the literature. 

First of all, we need the \emph{Helmholtz decomposition}, an observation which basically is a consequence of $\nabla\cdot(\nabla\times \Phi)=0$ and $\nabla\times (\nabla\phi) =0$ for smooth vector-fields $\Phi$ and smooth functions $\phi$.

\begin{theorem}\label{thm:Helmholtz}Let  $\Omega\subseteq \R^3$ be an open, bounded  Lipschitz domain. Then, there exist finite-dimensional spaces $\mathcal{H}_N(\Omega)$ and $\mathcal{H}_D(\Omega)$ such that
\[
     \Lp{2}(\Omega)^3 = \ran(\grad) \oplus \ran(\curl_0)\oplus \mathcal{H}_N(\Omega) = \ran(\grad_0) \oplus \ran(\curl)\oplus \mathcal{H}_D(\Omega).
\]
\end{theorem}
\begin{proof} By considering smooth, compactly supported vector fields only, it is not difficult to see that
\[
   \ran(\curl_0)\subseteq \ker(\div_0)\text{ and }   \ran(\grad_0)\subseteq \ker(\curl_0).
\]
Thus, for the orthogonal complements,
\[
   \ran(\grad)\subseteq \ker(\curl)\text{ and }\ran(\curl)\subseteq \ker(\div).
\]
Hence, from $\Lp{2}(\Omega)^3=\ran(\grad_0)\oplus \ker(\div) = \ran(\grad)\oplus \ker(\div_0)$, it follows that
\begin{align*}
   \Lp{2}(\Omega)^3 & = \ran(\grad_0) \oplus \ker(\div) \\
   & = \ran(\grad_0) \oplus \ran(\curl)\oplus \ker(\curl_0)\cap \ker(\div), \\
   \Lp{2}(\Omega)^3 & =\ran(\grad) \oplus \ker(\div_0) \\
      & = \ran(\grad) \oplus \ran(\curl_0)\oplus \ker(\curl)\cap \ker(\div_0).
\end{align*}
Since $\ker(\curl)\cap \ker(\div_0)\hookrightarrow \Lp{2}(\Omega)^3$ is compact if the former space is endowed with the sum of the graph norms of $\curl$ and $\div_0$, which, on $\ker(\curl)\cap \ker(\div_0)$, is equivalent to the $\Lp{2}(\Omega)^3$-norm, $\mathcal{H}_N(\Omega)\coloneqq \ker(\curl)\cap \ker(\div_0)$ is finite-dimensional. The same observation applies to $\mathcal{H}_D(\Omega)\coloneqq \ker(\curl_0)\cap \ker(\div)$. 
\end{proof}

For the proof of the main theorem of this section, we require three additonal results of differing complexity. The following result is a rather lengthy but elementary computation of block-operator matrices.

\begin{lemma}\label{lem:finitdim} Let $0<\alpha\leq\beta$ and $\mathcal{H}$ a Hilbert space. Let $(T_n)_{n\in\N},T$ in $\mathcal{F}(\alpha,\beta;\mathcal{H})$, $\mathcal{H}_0\subseteq \mathcal{H}$ a closed subspace, $\mathcal{K}\subseteq \mathcal{H}_1\coloneqq \mathcal{H}_0^\bot$. If $\mathcal{K}$ is finite-dimensional, then
\[
   [T_n \to T \text{ in }\tau(\mathcal{H}_0,\mathcal{H}_1)]\iff    [T_n \to T \text{ in }\tau(\mathcal{H}_0\oplus \mathcal{K} ,\mathcal{H}_1\cap \mathcal{K}^{\bot})]
\]
\end{lemma}

Finally, with a suitable localisation argument one also obtains the following theorem, where we use
\begin{multline*}
\sobH^1_\bot(\Omega)\coloneqq\{u\in \sobH^1(\Omega): \langle u,\chi\rangle_{ \Lp{2}(\Omega)}=0 \\\text{ for all }
  \chi\in \Lp{2}(\Omega)\text{ being constant on every connected component of }\Omega\}\text{.}
\end{multline*}
\begin{theorem}\label{thm:indbdycdt} Let $\Omega\subseteq\R^3$ be an open, bounded  Lipschitz domain, $0<\alpha\leq\beta$. Let $(a_n)_{n\in\N},a\in M(\alpha,\beta;\Omega)$. Then, the following conditions are equivalent:
\begin{enumerate}
  \item $a_n$ $\Htopo$-converges to $a$.
  \item for all $f\in \sobH^{-1}_{\bot}(\Omega)\coloneqq \sobH^1_\bot(\Omega)'$ and $u_n\in \sobH^1_\bot(\Omega)$ given as unique solutions of
\[
    \forall\phi \in \sobH^1_\bot(\Omega): \langle a_n \grad u_n, \grad\phi \rangle_{\Lp{2}(\Omega)^3} = f(\phi)
 \]
for $n\in\N$, we have $u_n\rightharpoonup u\in \sobH^1_\bot(\Omega)$ weakly and $a_n \grad u_n \rightharpoonup  a \grad u\in\Lp{2}(\Omega)^3$ weakly where $u\in \sobH_\bot^1(\Omega)$ satisfies
\[
     \forall\phi \in \sobH^1_\bot(\Omega):\langle a\grad u, \grad\phi \rangle_{\Lp{2}(\Omega)^3} = f(\phi).
 \] 
\end{enumerate}
\end{theorem}

Almost the same way as we established \Cref{thm:Hcon2}, it is possible to show the following result, where $\mathfrak{g}(\Omega)\coloneqq \ran(\grad)$.

\begin{theorem}\label{thm:indbdycdt2} Let $\Omega\subseteq\R^3$ be an open, bounded  Lipschitz domain, and $0<\alpha\leq\beta$. Let $(a_n)_{n\in\N},a\in M(\alpha,\beta;\Omega)$. Then, the following conditions are equivalent:
\begin{enumerate}
  \item $a_n$ $\tau(\mathfrak{g}(\Omega),\mathfrak{g}(\Omega)^\bot)$-converges to $a$.
  \item For all $f\in \sobH^{-1}_{\bot}(\Omega)$ and $u_n\in \sobH^1_\bot(\Omega)$ given as unique solutions of
\[
    \forall \phi \in \sobH^1_\bot(\Omega): \langle a_n \grad u_n, \grad\phi \rangle_{\Lp{2}(\Omega)^3} = f(\phi)
 \]
for $n\in\N$, we have $u_n\rightharpoonup  u\in \sobH^1_\bot(\Omega)$ weakly and $a_n \grad u_n \rightharpoonup  a \grad u\in\Lp{2}(\Omega)^3$ weakly where $u\in \sobH_\bot^1(\Omega)$ satisfies
\[
   \forall \phi \in \sobH^1_\bot(\Omega):  \langle a\grad u, \grad\phi \rangle_{\Lp{2}(\Omega)^3} = f(\phi).
 \] 
\end{enumerate}
\end{theorem}

Finally we can prove the homogenisation result for Maxwell's equations.

\begin{proof}[Proof of \Cref{thm:homomaxwell}]
 By \Cref{thm:abstractSchur} we need to show that 
 \[
    T_n(\lambda) \coloneqq \begin{pmatrix} \lambda \varepsilon_n +
    \sigma_n & 0 \\ 0 & \mu_n \end{pmatrix} \to T\coloneqq \begin{pmatrix}  \varepsilon(\lambda)& 0 \\ 0 & \mu \end{pmatrix}
 \]in $\tau(\ker(A),\ran(A))$ for $A = \begin{psmallmatrix} 0 & -\curl \\ \curl_0 & 0 \end{psmallmatrix} $, which satifsfies the compactness assumption by the Picard--Weber--Weck selection theorem and is skew-selfadjoint by \Cref{prop:adjblock}. Next, by \Cref{thm:Helmholtz},
 \begin{multline*}
    \ker(A) = \ker(\curl_0)\oplus \ker(\curl) = \ran(\curl)^\bot \oplus \ran(\curl_0)^\bot \\ = \Big(\ran(\grad_0)\oplus \mathcal{H}_D(\Omega)\Big)\oplus \Big(\ran(\grad)\oplus \mathcal{H}_N(\Omega)\Big).
 \end{multline*}
 By \Cref{thm:indbdycdt}, (sequential) convergence in $\tau(\ker(A),\ran(A))$ is equivalent to the one in $\tau(\mathfrak{g}_0(\Omega)\oplus \mathfrak{g}(\Omega),\mathfrak{g}_0(\Omega)^\bot\oplus \mathfrak{g}(\Omega)^\bot)$. Hence, the assertion follows from \Cref{thm:indbdycdt}, \Cref{thm:indbdycdt2}, and \Cref{lem:finitdim}.
\end{proof}

\section{Comments}
 We refer to \cite[Chapter~2]{SeTrWa22} for a quick round-up concerning the adjoint operator for unbounded operators. \Cref{thm:abstractSchur} and its proof are based on
 \cite[Theorem~3.8]{BuErWa24} and \cite[Lemma~6.4]{BuSkWa24}. Schur compact sets are discussed in \cite[Section~5]{Wa18}.
  The operator-theoretic setting for partial differential equations used in \Cref{sec:ex} is the one of evolutionary equations as introduced by Picard in 2009, see \cite{Pi09}. A relatively slow paced introduction to this field with many examples and further aspects can be found in the monograph \cite{SeTrWa22}.  \Cref{thm:FourLapWOT} is a part of \cite[Theorem~1.3]{BuErWa24}, and the thermoelastic model presented here a simplified version of \cite[Section 4.3]{Pi09}. For (weak) Lipschitz domains, the Picard--Weber--Weck theorem (including detailed proofs), and the Helmholtz decomposition see, e.g., \cite{Pi82,Pi84,BPS16,PW22}.
 \Cref{lem:finitdim} is one of the main results in \cite{Wau25}. \Cref{thm:indbdycdt} and \Cref{thm:indbdycdt2}
 fundamentally rely on \cite[Lemma~10.3]{Ta09}, see also \cite[Section~4]{BFSW24}.
\begin{takehomes}
 \begin{enumerate}
   \item[(a)] Many PDEs can be written in the form $(T+A)u=f$ for some skew-selfadjoint $A$ and bounded $T$ with $\Re T\geq c$.
      \item[(b)] Convergence results for operator equations require a compactness property for $A$. Thus, for a nice representation of the limit, bounded domains are required -- `No general homogenisation of the universe' (quote by R.~Picard).
         \item[(c)] $\Htopo$-convergence is independent of the boundary conditions.
            \item[(d)] Abstract results can be easily applied to Thermoelasticty and Maxwell's equations.
               \item[(e)] The Schur topology offers an abstract concept simplifying the approach to concrete PDEs without unnecessary clutter.
 \end{enumerate}
\end{takehomes}

\section{Exercises}

\begin{exercise}\label{exer:3.0}
Show that the product topology on the countable product of metrisable spaces is metrisable again.
\end{exercise}
\begin{exercise}\label{exer:3.05}
Let $\mathcal{H}$ be a Hilbert space, $\mathcal{H}_0\subseteq \mathcal{H}$ a closed subspace, $\mathcal{H}_1\coloneqq \mathcal{H}_0^\bot$, and $0<\alpha\leq\beta$.
Show that there exists a $\gamma\in (0,\infty)^{2\times 2}$ such that $\mathcal{F}(\alpha,\beta;\mathcal{H})\subseteq\mathcal{F}(\gamma; \mathcal{H}_0,\mathcal{H}_1)$.
\end{exercise}
\begin{exercise}\label{exer:3.1} Prove \Cref{prop:adjblock}.
\end{exercise}

\begin{exercise}\label{exer:3.2} Prove \Cref{lem:inv}.

Hint: Use the expression of the inverse of an operator as a block matrix operator from the proof of \Cref{thm:HnlH}.
\end{exercise}

\begin{exercise}${}^*$\label{exer:3.3} Prove \Cref{thm:indbdycdt2}.
\end{exercise}
\begin{exercise}
 Let $\mathcal{H}$ be a Hilbert space, $\mathcal{H}_0\subseteq \mathcal{H}$ a closed subspace, $\mathcal{H}_1\coloneqq \mathcal{H}_0^\perp$, $(T_n)_{n\in\N}$, $T$ in $\mathcal{M}(\mathcal{H}_0,\mathcal{H}_1)$ with $\Re T_n\geq \alpha$ and $\Re T_n^{-1}\geq 1/\beta$ for all $n\in \N$ and some $0<\alpha\leq\beta$. If $T_n \to T$ in the strong operator topology, show that then $T_n\to T$ in $
\tau(\mathcal{H}_0,\mathcal{H}_1)$.
\end{exercise}

\end{document}